\documentclass[reqno]{amsart}
\usepackage{amsmath,amsfonts,amssymb,color}
\newtheorem{teor}{Theorem}[section]

\newtheorem{prop}[teor]{Proposition}

\theoremstyle{definition}
\newtheorem{defi}[teor]{Definition}
\newtheorem{eje}[teor]{Example}

\newtheorem{nota}[teor]{Remark}

\numberwithin{equation}{section}

\newcommand{\R}{\mathbb{R}}

\newcommand{\W}{\Omega}
\newcommand{\w}{\omega}

\newcommand{\dd}{\textsf{d}}

\newcommand{\wit}{\widetilde}

\newcommand{\lsm}{\left[\!\begin{smallmatrix}}
\newcommand{\rsm}{\end{smallmatrix}\!\right]}
\newcommand{\nbd}{\nobreakdash}
\newcommand{\des}{\displaystyle}

\DeclareMathOperator{\Int}{Int} 
 \DeclareMathOperator{\spa}{span}
\begin{document}
\title[Persistence of almost periodic Nicholson systems]
{Is uniform persistence a robust property in almost periodic models? A well-behaved family: almost periodic Nicholson systems}
\author[R. Obaya]{Rafael Obaya}
\author[A.M. Sanz]{Ana M. Sanz}
\address[R. Obaya]{Departamento de Matem\'{a}tica
Aplicada, E. Ingenier\'{\i}as Industriales, Universidad de Valladolid,
47011 Valladolid, Spain, and member of IMUVA, Instituto de Investigaci\'{o}n en
Matem\'{a}ticas, Universidad de Valladolid.}
 \email{rafoba@wmatem.eis.uva.es}
\address[A.M. Sanz]{Departamento de Did\'{a}ctica de las Ciencias Experimentales, Sociales y de la Matem\'{a}tica,
Facultad de Educaci\'{o}n, Universidad de Valladolid, 34004 Palencia, Spain,
and member of IMUVA, Instituto de Investigaci\'{o}n en  Mate\-m\'{a}\-ti\-cas, Universidad de
Valladolid.} \email{anasan@wmatem.eis.uva.es}
\thanks{The authors were partly supported by Ministerio de Econom\'{\i}a y Competitividad
under project MTM2015-66330, and the European Commission under project H2020-MSCA-ITN-2014}
\date{}
\begin{abstract}
Using techniques of non-autonomous dynamical systems, we completely characterize the persistence properties of an almost periodic Nicholson system in terms of some numerically computable exponents. Although similar results hold  for a class of cooperative and sublinear models, in the general non-autonomous setting one has to consider persistence as a collective property of the family of systems over the hull: the reason is that uniform persistence is not a robust property in models given by almost periodic differential equations.
\end{abstract}
\renewcommand{\subjclassname}{\textup{2010} Mathematics Subject Classification}
\maketitle
\section{Introduction}\label{sec-intro}
In the field of non-autonomous differential equations with a certain recurrent variation in time (such as almost periodicity) it is a common approach  to consider a system not just as a system on its own, but as a member of a whole family of systems: the one obtained through the so-called {\em hull construction\/} (for instance, see Johnson~\cite{john}, or Section~\ref{sec-preli}). The reason it that then the theory of non-autonomous dynamical systems or skew-product semiflows applies. In this context the study of the robustness of some dynamical properties has been a common issue. For instance, for linear differential systems the existence  of an exponential dichotomy is a robust property, meaning that if a given system has an exponential dichotomy, then the family of systems over the hull also has an exponential dichotomy (see Sacker and Sell~\cite{sase} in the finite-dimensional case). In this paper we focus on the property of uniform persistence. The question whether persistence is robust in autonomous equations has long been considered (for instance see Hofbauer and Schreiber~\cite{hosc}).
\par
Persistence is a dynamical property which has a great interest in mathematical modelling, in areas such as  biological population dynamics, epidemiology, ecology or neural networks. In the field of dynamical systems, different notions of persistence have been introduced, with the general meaning that in the long run the trajectories of the system place themselves above a prescribed region of the phase space. In many applications this region is determined by the null solution, so that, roughly speaking, uniform persistence means that solutions eventually become uniformly strongly positive.
\par
This paper is heavily motivated by the recent papers by Novo et al.~\cite{noos7} and Obaya and Sanz~\cite{obsa}, where the authors determine sufficient conditions for the uniform and strict persistence,  respectively, of families of non-autonomous cooperative systems of ordinary differential equations  (ODEs for short) and delay functional differential equations (FDEs for short) over a minimal base flow. The concepts of persistence are given as a collective property of the whole family. In the linear case the conditions given are not only sufficient but also necessary. Moreover, as a nontrivial application of the results,  in~\cite{obsa} a spectral characterization of the persistence properties of families of almost periodic Nicholson systems has been given. We refer the reader unfamiliar with Nicholson systems to Section~\ref{sec-persistence} and we just cite some very recent related works such as  Berezansky et al.~\cite{bebrid}, Liu and Meng~\cite{lime}, Faria~\cite{faria},  Wang~\cite{wang}, Faria and R\"{o}st~\cite{faro} and Faria et al.~\cite{faos}.
\par
At this point it is natural to wonder whether a characterization of the persistence properties, similar to that in~\cite{obsa}, can be given when one considers not a whole family, but only an individual almost periodic  Nicholson system.  In other words, if we can decide on the persistence properties of a given Nicholson system in terms of some computable items coming out of the system. In this case, the answer is in the affirmative, as it is stated in Theorem~\ref{teor Nicholson system}. But, as the reader might expect, the answer comes after a nice transfer of the property of persistence from the individual system to the whole family of systems over the  hull, to which the characterization given in~\cite{obsa} applies, so that it only remains to check that the exponents involved can in fact be computed just out of the individual system. This is clearly the most desirable situation, but, is it the general situation?
\par
The last question demands to solve the underlying problem on the robustness of uniform persistence. We show that this dynamical property is not robust in almost periodic ODEs or delay FDEs, meaning that in general it is not transferred from an individual almost periodic system to the family of systems over the hull. Besides, when the transfer fails to happen in models given by almost periodic cooperative and linear or sublinear ODEs or delay FDEs, the set of systems which do not gain the property is big, both from a topological and from a measure theory point of view. This means that in the non-robust situation it is highly improbable that we can experimentally or  numerically detect uniform persistence.
\par
The previous fact naturally raises a discussion on the proper definition of uniform persistence in the non-autonomous field, which in the general case results in the convenience of adopting the collective formulation given in~\cite{noos7}.
Notwithstanding, still uniform persistence is a robust property in some other models of real life processes, as those given by cooperative and sublinear ODEs or delay FDEs with some strongly positive bounded solution. In this case the consideration of an individual system is enough in what refers to the study of its persistence properties, and we  can characterize persistence through a set of numerically computable objects.
\par
To emphasize the importance of the latter fact from the point of view of applications, recall that there is a long  tradition in the study of monotone and sublinear, concave or convex
semiflows generated by families of differential equations, clearly motivated by their frequent appearence in mathematical modelling, apart from the theoretical interest itself.
 The works by Shen and Yi~\cite{shyi5,shyi},
Zhao~\cite{zhaox,zhaox2}, Mierczy{\'n}ski and Shen~\cite{mish}, Novo et al.~\cite{nooslondon} and N\'{u}\~{n}ez et al.~\cite{nuos2,nuos3,nuos4}
contain some significative examples of the application of
dynamical arguments to analyze non-autonomous differential equations modelling processes  in engineering, biology and ecology, among other
branches of science.
\par
We finally briefly describe the organization and main results of the paper. Section~\ref{sec-preli} contains some necessary preliminaries in order to make the paper reasonably self-contained. In particular, the definitions of skew-product semiflow and of a continuous separation for linear monotone skew-product semiflows are included, as they are important tools in our results.
\par
Section \ref{sec-persistence} is devoted to the persistence properties of an almost periodic Nicholson system. We prove that a persistence  property of a particular Nicholson system is transferred to the family of systems over the hull. As a consequence, taking advantage of the results in~\cite{obsa} for such a family, we characterize the persistence properties of an almost periodic Nicholson system by means of some numerically computable exponents. The same results hold for other models considered in the literature, as the one for hematopoiesis given in Mackey and Glass~\cite{magl}.
\par
Finally, in Section \ref{sec-ejemplo}, being aware of the dynamical complexity that  scalar almost periodic linear differential equations can exhibit (see Poincar\'{e}~\cite{poin} or Johnson~\cite{john}) in contrast with the cases of autonomous or periodic equations,  we offer a concrete example of an almost periodic linear scalar equation for which uniform persistence is not a robust property, that is, it is not transferred to the family of systems over the hull. This phenomenon can also appear in  higher dimensional linear and nonlinear systems. Despite this fact, still we can determine a wide class of almost periodic cooperative and sublinear ODEs and delay FDEs systems for which things go nicely in what refers to persistence, i.e., as in the Nicholson's case. The additional condition needed is the existence of a strongly positive bounded solution.
\section{Some preliminaries}\label{sec-preli}
In this section we include some preliminaries of
topological dynamics for non-autonomous dynamical systems, as well as some classes of almost periodic systems of ODEs and finite-delay FDEs which will be considered.
\par
Let $(\W,d)$ be a compact
metric space. A real {\em continuous flow\/} $(\W,\sigma,\R)$ is
defined by a continuous map $\sigma: \R\times \W \to  \W,\;
(t,\w)\mapsto \sigma(t,\w)$ satisfying
\begin{enumerate}
\renewcommand{\labelenumi}{(\roman{enumi})}
\item $\sigma_0=\text{Id},$
\item $\sigma_{t+s}=\sigma_t\circ\sigma_s$ for each $s$, $t\in\R$,
\end{enumerate}
where $\sigma_t(\w)=\sigma(t,\w)$ for all $\w \in \W$ and $t\in \R$.
The set $\{ \sigma_t(\w) \mid t\in\R\}$ is called the {\em orbit\/}
of the point $\w$. We say that a subset $\W_1\subset \W$ is {\em
$\sigma$-invariant\/} if $\sigma_t(\W_1)=\W_1$ for every $t\in\R$. A
subset $\W_1\subset \W$ is called {\em minimal \/} if it is compact,
$\sigma$-invariant and it does not contain properly any other
compact $\sigma$-invariant set. Based on Zorn's lemma, every compact
and $\sigma$-invariant set contains a minimal subset. Furthermore, a
compact $\sigma$-invariant subset is minimal if and only if every
orbit is dense. We say that the continuous flow $(\W,\sigma,\R)$ is
{\em recurrent\/} or {\em minimal\/} if $\W$ is minimal.
The flow $(\W,\sigma,\R)$ is {\em almost periodic\/} if the family of maps $\{\sigma_t\}_{t\in \R}:\W\to\W$ is uniformly equicontinuous  on $\W$, that is, for every
$\varepsilon
> 0 $ there is a $\delta >0$ such that, if $\w_1$, $\w_2\in\W$
with $d(\w_1,\w_2)<\delta$, then
$d(\sigma(t,\w_1),\sigma(t,\w_2))<\varepsilon$ for every $t\in \R$.
\par
A finite regular measure defined on the Borel sets of $\W$ is called
a Borel measure on $\W$. Given $\mu$ a normalized Borel measure on
$\W$, it is {\em $\sigma$-invariant\/} (or {\em invariant under\/}
$\sigma$) if $\mu(\sigma_t(\W_1))=\mu(\W_1)$ for every Borel subset
$\W_1\subset \W$ and every $t\in \R$. It is {\em ergodic\/}  if, in
addition, $\mu(\W_1)=0$ or $\mu(\W_1)=1$ for every
$\sigma$-invariant subset $\W_1\subset \W$.
We  denote by $\mathcal{M}_{\text{inv}}(\W,\sigma,\R)$ the set of
all positive and normalized $\sigma$-invariant measures on $\W$. The
Krylov\nbd-Bogoliubov theorem asserts that
$\mathcal{M}_{\text{inv}}(\W,\sigma,\R)$ is nonempty when $\W$ is a
compact metric space. The extremal points of the convex and weakly
compact set $\mathcal{M}_{\text{inv}}(\W,\sigma,\R)$ are the ergodic
measures, and thus the set of ergodic
measures $\mathcal{M}_{\text{erg}}(\W,\sigma,\R)$ is nonempty. We
say that $(\W,\sigma,\R)$ is {\em uniquely ergodic\/} if it has a
unique normalized invariant measure, which is then necessarily
ergodic. A minimal and almost periodic flow $(\W,\sigma,\R)$ is uniquely ergodic.
\par
Let $\R_+=\{t\in\R\,|\,t\geq 0\}$. Given a continuous compact flow $(\W,\sigma,\R)$ and a
complete metric space $(X,\dd)$, a continuous {\em skew-product semiflow\/} $(\W\times
X,\tau,\,\R_+)$ on the product space $\W\times X$ is determined by a continuous map
\begin{equation}\label{skew}
 \begin{array}{cccl}
 \tau \colon  &\R_+\times\W\times X& \longrightarrow & \W\times X \\
& (t,\w,x) & \mapsto &(\w{\cdot}t,u(t,\w,x))
\end{array}
\end{equation}
 which preserves the flow on $\W$, denoted by $\w{\cdot}t=\sigma(t,\w)$ and referred to as the {\em base flow\/}.
 The semiflow property means that
\begin{enumerate}
\renewcommand{\labelenumi}{(\roman{enumi})}
\item $\tau_0=\text{Id},$
\item $\tau_{t+s}=\tau_t \circ \tau_s\;$ for all  $\; t$, $s\geq 0\,,$
\end{enumerate}
where again $\tau_t(\w,x)=\tau(t,\w,x)$ for each $(\w,x) \in \W\times X$ and $t\in \R_+$.
This leads to the so-called semicocycle property,
\begin{equation*}
 u(t+s,\w,x)=u(t,\w{\cdot}s,u(s,\w,x))\quad\mbox{for $s,t\ge 0$ and $(\w,x)\in \W\times X$}\,.
\end{equation*}
\par
The set $\{ \tau(t,\w,x)\mid t\geq 0\}$ is the {\em semiorbit\/} of
the point $(\w,x)$. A subset  $K$ of $\W\times X$ is {\em positively
invariant\/}, or  $\tau$-{\em invariant\/}, if $\tau_t(K)\subseteq K$
for all $t\geq 0$.  A compact $\tau$-invariant set $K$ for the
semiflow  is {\em minimal\/} if it does not contain any nonempty
compact $\tau$-invariant set  other than itself. The restricted
semiflow over a compact and $\tau$-invariant set $K$ admits a {\em
flow extension\/} if there exists a continuous flow
$(K,\wit\tau,\R)$ such that $\wit \tau(t,\w,x)=\tau(t,\w,x)$ for all
$(\w,x)\in K$ and $t\in\R_+$.
\par
Whenever a semiorbit $\{\tau(t,\w_0,x_0)\mid t\ge 0\}$ is relatively
compact, one can consider the {\em omega-limit set\/} of
$(\w_0,x_0)$, denoted by $\mathcal{O}(\w_0,x_0)$ and formed by the
limit points of the semiorbit as $t\to\infty$, that is, the pairs
$(\w,x)=\lim_{n\to\infty} \tau(t_n,\w_0,x_0)$ for some sequence
$t_n\uparrow \infty$. The set $\mathcal{O}(\w_0,x_0)$ is then a
nonempty compact connected and $\tau$-invariant set.
\par
The reader can find in  Ellis~\cite{elli}, Sacker and
Sell~\cite{sase}, Shen and Yi~\cite{shyi} and references therein, a
more in-depth survey on topological dynamics.
\par
In this paper we will sometimes work under differentiability assumptions.
When $X$ is a Banach space, the
semiflow~\eqref{skew} is said to be of class $\mathcal C^1$ when $u$
is assumed to be of class $\mathcal C^1$ in $x$, meaning that
$u_x(t,\w,x)$ exists for any $t>0$ and any $(\w,x)\in\W\times X$ and
for each fixed $t>0$, the map $(\w,x)\mapsto u_x(t,\w,x)\in \mathcal
L(X)$ is continuous in a neighborhood of any compact set $K\subset
\W\times X$; moreover, for any $z\in X$, $\lim_{\,t\to
0^+}u_x(t,\w,x)\,z=z $ uniformly for $(\w,x)$ in compact sets of
$\W\times X$.
\par
In that case, whenever $K\subset \W\times X$ is a compact positively
invariant set, we can define a continuous  linear skew-product
semiflow called the {\em linearized skew-product semiflow\/}
of~\eqref{skew} over $K$,
\begin{equation*}
 \begin{array}{cccl}
 L: & \R_+\times K \times X& \longrightarrow & K \times X\\
& (t,(\w,x),z) & \mapsto &(\tau(t,\w,x),u_x(t,\w,x)\,z)\,.
\end{array}
\end{equation*}
We note that $u_x$ satisfies the linear semicocycle property
\begin{equation*}
u_x(t+s,\w,x)=u_x(t,\tau(s,\w,x))\,u_x(s,\w,x)\,,\quad
s,t\in\R_+\,,\;\, (\w,x)\in K.
\end{equation*}
\par
We now introduce Lyapunov exponents. For $(\w, x)\in K$ we
denote by  $\lambda(\w,x)$ the {\em Lyapunov exponent}\/  defined as
\begin{equation*}
\lambda(\w,x)=\limsup_{\,t\to\infty}
\frac{\log\|u_x(t,\w,x)\|}{t}\,.
\end{equation*}
The number $\lambda_K=\sup_{\,(\w,x)\in K} \lambda(\w,x)$ is called
the {\em upper Lyapunov exponent\/} of~K.
\par
Also, reference  will be made to  monotone, and to monotone and concave or monotone and sublinear  skew-product semiflows. When the state space $X$ is a strongly ordered Banach space, that is, there is a closed convex solid cone of nonnegative vectors $X_+$ with a nonempty interior, then, a (partial) {\em strong order relation\/} on $X$ is
defined by
\begin{equation*}
\begin{split}
 x\le y \quad &\Longleftrightarrow \quad y-x\in X_+\,;\\
 x< y  \quad &\Longleftrightarrow \quad y-x\in X_+\;\text{ and }\;x\ne y\,;
\\  x\ll y \quad &\Longleftrightarrow \quad y-x\in \Int X_+\,.\qquad\quad\quad~
\end{split}
\end{equation*}
The positive cone is usually assumed to be {\em normal} (see Amann~\cite{aman} for more details). In this situation, the skew-product semiflow~\eqref{skew}
is {\em monotone\/} if
\begin{equation*}
 u(t,\w,x)\le u(t,\w,y)\,\quad \text{for\, $t\ge 0$, $\w\in\W$ \,and\,
 $x,y\in X$ \,with\, $x\le y$}\,.
\end{equation*}
\par
A monotone skew-product semiflow
is said to be {\em concave\/} if for any $t\ge 0$, $\w\in\W$, $x\leq y$ and
$\lambda\in[0,1]$,
\begin{equation*}
 u(t,\w,\lambda\,y+(1-\lambda)\,x)\ge \lambda
 \,u(t,\w,y)+(1-\lambda)\,u(t,\w,x)\,,
\end{equation*}
and a skew-product semiflow with the
positivity property (that is, $\W\times X_+$ is $\tau$-invariant) is
{\em sublinear\/} if
\[
u(t,\w,\lambda\,x)\ge \lambda\,u(t,\w,x)\,\quad \mbox{for any}\;\,
t\ge 0\,,\; \w\in\W\,,\; x\in X_+\; \hbox{and } \lambda\in[0,1]\,.
\]
\par
The dynamical description of monotone and sublinear and of monotone and concave skew-product semiflows found respectively in N\'{u}\~{n}ez et al.~\cite{nuos2} and~\cite{nuos4} will be useful in this work.
\par
We now include the definitions of a continuous separation in the classical terms of Pol\'{a}\v{c}ik and Tere\v{s}\v{c}\'{a}k~\cite{pote} in the discrete case, generalized by Shen and Yi~\cite{shyi} to the continuous case, and of a continuous separation of type II in the terms introduced by Novo et al.~\cite{noos6}.
A continuous linear and monotone skew-product semiflow over a minimal base flow $(\W,{\cdot},\R)$ and a strongly ordered Banach space $X$,
\begin{equation*}
 \begin{array}{cccl}
  L: &   \R_+\times \W \times X& \longrightarrow & \W \times X\\
&(t,\w,v) & \mapsto &(\w{\cdot}t,\Phi(t,\w)\,v)\,,
\end{array}
\end{equation*}
which satisfies that for each $t>0$ the map $\W\to \mathcal L(X)$,
$\w\mapsto \Phi(t,\w)$ is continuous,
 is said to admit a {\em continuous
separation\/} if there are families of subspaces
$\{X_1(\w)\}_{\w\in \W}$ and $\{X_2(\w)\}_{\w\in
\W}\subset X$ satisfying the following properties:
\begin{itemize}
\item[(S1)] $X=X_1(\w)\oplus X_2(\w)$  and $X_1(\w)$, $X_2(\w)$ vary
    continuously in $\W$;
 \item[(S2)] $X_1(\w)=\spa\{ v(\w)\}$, with $v(\w)\gg 0$ and
     $\|v(\w)\|=1$ for any $\w\in \W$;
\item[(S3)] $X_2(\w)\cap X_+=\{0\}$ for any $\w\in \W$;
\item[(S4)] for any $t>0$,  $\w\in \W$,
\begin{align*}
\Phi(t,\w)X_1(\w)&= X_1(\w{\cdot}t)\,,\\
\Phi(t,\w)X_2(\w)&\subset X_2(\w{\cdot}t)\,;
\end{align*}
\item[(S5)] there are $M>0$, $\delta>0$ such that for any $\w\in \W$, $z\in
    X_2(\w)$ with $\|z\|=1$ and $t>0$,
\begin{equation*}
\|\Phi(t,\w)\,z\|\leq M \,e^{-\delta t}\|\Phi(t,\w)\,v(\w)\|\,.
\end{equation*}
\end{itemize}
When property (S3) does not hold,  but still it is replaced by (S3)' below, then the continuous separation is said to be of type II.
\begin{itemize}
\item[(S3)'] there exists a $T>0$ such that if for some $\w\in
\W$ there is a
    $z\in X_2(\w)$ with $z>0$, then $\Phi(t,\w)\,z=0$ for any $t\geq T$.
\end{itemize}
\par
To finish this section, we include a general class of almost periodic ODEs and delay FDEs whose solutions can be immersed into a skew-product semiflow by using the so-called hull construction. In order to build the hull, admissibility is the key property. A function $f\in C(\R\times\R^m,\R^n)$ is said to be {\em admissible\/} if for any
compact set $K\subset \R^m$, $f$ is bounded and uniformly continuous
on $\R\times K$. Recall that a continuous function $f:\R\to\R$ is {\em almost periodic\/} if for any $\varepsilon>0$ the $\varepsilon$-translate set of $f$, $T_\varepsilon(f)=\{r\in \R\mid |f(t+r)-f(t)|< \varepsilon\;\,\text{for any}\; t\in \R\}$ is a relatively dense set in $\R$, that is, there exists an $l>0$ such that any interval of length $l$ has a nonempty intersection with the set $T_\varepsilon(f)$.
\par
We consider  $n$-dimensional systems of ODEs given by a uniformly almost periodic function
$f\colon \R\times\R^n\to\R^n$ (that is, $f$ is admissible and  $f(t,y)$ is almost periodic in $t$ for any $y\in\R^n$), of class $\mathcal C^1$ with respect to $y$ and such that its first order derivatives $\partial
f/\partial y_i$, $i=1,\ldots,n$ are admissible,
\begin{equation}\label{ode}
y'(t)=f(t,y(t))\,,\quad t\in\R\,;
\end{equation}
and $n$-dimensional systems of finite-delay differential equations with
a fixed delay, which we take to be $1$, given by a uniformly almost periodic function
$f:\R\times\R^n\times \R^n\to\R^n$, with the same regularity and
admissibility conditions as before,
\begin{equation}\label{delay}
y'(t)=f(t,y(t),y(t-1))\,, \quad t>0\,.
\end{equation}
\par
In both of the previous situations, let $\W$ be the {\em hull\/} of $f$, that is, the closure for the topology of uniform convergence on compacta of
the set of $t$-translates of $f$, $\{ f_t \mid t\in\R\}$ with $f_t(s,z)=f(t+s,z)$
for $s\in \R$ and $z\in\R^n$ or $\R^{2n}$, adequate to
each case. The translation map $\R\times \W\to \W$,
$(t,\w)\mapsto\w{\cdot}t$ given by $\w{\cdot}t(s,z)= \w(s+t,z)$ ($s\in \R$ and $z\in\R^n$ or $\R^{2n}$) defines a
continuous flow $\sigma$ on the compact metric space $\W$, which is minimal and almost periodic, and thus uniquely ergodic.
Each function $\w\in\W$ has the
same regularity and admissibility properties as those of $f$, and
$F\colon \W\times \R^p\to\R^n$, $(\w,z)\mapsto \w(0,z)$ (with $p=n$ or $p=2n$) can be looked at as the unique continuous
extension of $f$ to its hull. Thus, in each case we can consider the
family of $n$-dimensional systems over the hull, which we write for
short as:
\begin{equation}\label{odefamily}
 y'(t)=F(\w{\cdot}t,y(t))\,,\quad\w\in\W
\end{equation}
for the ODEs case; and
\begin{equation}\label{delayfamily}
y'(t)=F(\w{\cdot}t,y(t),y(t-1))\,,\quad \w\in\W
\end{equation}
in the delay case, whose solutions  induce a forward
dynamical system of skew-product type~\eqref{skew} (in principle only locally-defined) on the product $\W\times X$. Namely, in the ODEs case we take $X=\R^n$ endowed with the norm
$\|x\|=|x_1|+\cdots+|x_n|$ for $x\in \R^n$, with the normal positive cone $\R^n_+=\{y\in \R^n\mid y_i\geq 0 \;\text{for}\;i=1,\ldots,n\}$ which induces a (partial) strong ordering on $\R^n$ defined componentwise, and  $u(t,\w,x)$ is the value of the solution of system~\eqref{odefamily} for $\w$ at time $t$  with initial condition $x\in X$.   In the delay case we take $X=C([-1,0],\R^n)$ with the norm
$\|\varphi\|=\|\varphi_1\|_\infty+\ldots+\|\varphi_n\|_\infty$ for $\varphi\in X$, and the positive cone $X_+=\{\varphi \in X \mid \varphi(s)\geq 0 \,\text{ for all }\, s\in[-1,0]\}$
which is normal and has nonempty interior $\Int X_+=\{\varphi\in X \mid
\varphi(s)\gg 0 \,\text{ for all }\, s\in[-1,0]\}$. In this case $u(t,\w,x)=y_t(\w,x)$, which is defined as $y_t(\w,x)(s)=y(t+s,\w,x)$ for $s\in [-1,0]$ for the solution $y(t,\w,x)$ of system~\eqref{delayfamily} for $\w$ at time $t$ with initial condition $x\in X$. In both cases, a bounded solution gives rise to a relatively compact semiorbit, so that the omega-limit set is well-defined. (For the standard theory of delay FDEs see Hale and
Verduyn Lunel~\cite{have}.)
\par
In order that the skew-product semiflow be monotone, $f$ is required to be {\em cooperative\/}. Under the former regularity assumptions, the cooperative condition for system~\eqref{ode} is written as
\[
 \frac{\partial f_i}{\partial y_j}(t,y) \geq 0 \;\, \text{ for } i\not= j\,,\; \text{ for any }  (t,y)\in\R\times\R^n\,;
\]
and for system~\eqref{delay} it is written as
\[
\frac{\partial f_i}{\partial y_j}(t,y,w) \geq 0 \;\,
 \text{ for } i\not= j \;\, \text{ and }\;  \frac{\partial f_i}{\partial w_j}(t,y,w)\ge 0\,\;\text{ for any}\;\, i, j\,,
\]
for any
$(t,y,w)\in\R\times\R^n\times\R^n$. If the initial system is cooperative, then so are all the systems over the hull.
By standard arguments of comparison of solutions (for instance, see Smith~\cite{smit}), this condition implies that the induced semiflow is monotone (on its domain of definition).
\par
Finally,  system~\eqref{ode} (resp.~system~\eqref{delay}) is (order) {\em concave\/} if
\[
f(t,\lambda\, y+(1-\lambda)\,x)\ge \lambda\,
f(t,y)+(1-\lambda)\,f(t,x)\]
for any $t\in\R$, $\lambda\in[0,1]$
and $x,y\in\R^{n}$ (resp. $x,y\in\R^{2n}$) with $x\le y$; and it is {\em sublinear\/} if
\[
f(t,\lambda\, y)\ge \lambda\,
f(t,y)\]
 for any $t\in\R$, $\lambda\in[0,1]$
and $y\in\R^{n}_+$ (resp.~$y\in\R^{2n}_+$). Note that, under the assumption that $0$ is a solution, the concave condition actually implies the sublinear condition. Once more, by standard arguments of comparison of solutions (see Smith~\cite{smit}), if the system is cooperative and concave\//\/sublinear, the induced semiflow is monotone and concave\//\/sublinear (on its domain of definition).
\par
\section{Persistence properties of almost periodic Nicholson systems}\label{sec-persistence}
For the reader unfamiliar with Nicholson systems, the most remarkable facts are the following. In 1954 Nicholson~\cite{nich} published experimental data on the behaviour of the population of the Australian sheep-blowfly. Then, Gurney et al.~\cite{gubl} studied the scalar delay equation
\begin{equation*}
x'(t)=-\mu\,x(t)+p\,x(t-\tau)\,e^{-\gamma\,x(t-\tau)}\,,
\end{equation*}
which was called the Nicholson's blowflies equation, as it suited the experimental data reasonably well. Here, $\mu,\,p,\,\gamma$ and $\tau$ are positive constants with a biological interpretation. In particular the delay $\tau$ stands for the maturation time of the species.  The interest of Nicholson himself was in the existence of oscillatory solutions for the behaviour of the adult population. Later on, many authors have determined different relations of the coefficients so as to have global asymptotic stability of the nontrivial positive steady state solution, though the general problem is still not closed (see Smith~\cite{smit} and  Berezansky et al.~\cite{bebrid}). Concerned with stability, persistence or existence of certain kind of solutions, among other dynamical issues, some generalizations and modifications of the Nicholson equation have also been considered.
\par
More recently, Nicholson systems have been introduced, as they fit models for one single species in an environment with a patchy structure or for multiple biological species. Taking time-dependent coefficients and adding a patch-structure helps to model the seasonal variation of the environment as well as the presence of a heterogeneous environment, so that there are $n$ patches in which the individuals can live, each of them determined by different climate, different food resources, and so on. In this way,
 the distribution of the population is influenced by the growth and death rates of the populations in each patch and migrations among patches. Also the maturation time is assumed to be possibly different in each patch.
\par
In this section we consider an almost periodic noncooperative system with delay which is among the family of Nicholson systems. Namely, we consider an $n$-dimensional system of delay FDEs with a patch-structure ($n$ patches) and a nonlinear term of Nicholson type, which is able to reflect an almost periodic temporal variation in the environment,
\begin{equation}\label{nicholson delay}
y_i'(t)=-\wit d_i(t)\,y_i(t) +\des \sum_{j=1}^n \wit a_{ij}(t)\,y_j(t) + \wit\beta_{i}(t)\,y_i(t-\tau_{i})\,e^{-\wit c_i(t)\,y_i(t-\tau_{i})}\,,\quad t\geq 0\,,
\end{equation}
for $i=1,\ldots,n$. Here $y_i(t)$ denotes the density of the population in patch $i$ at time $t\geq 0$, and
$\tau_i>0$ is the maturation time in that patch. We consider the delay system together with an initial condition, which is given by a map $\varphi=(\varphi_1,\ldots, \varphi_n)\in C([-\tau_1,0])\times \ldots \times C([-\tau_n,0])$, which is assumed to be nonnegative in all components, due to the implicit biological meaning. Let us denote by $y(t,\varphi)$ the solution of this problem, whenever defined.
\par
We make the following assumptions on the coefficient functions:
\begin{itemize}
\item[(a1)] $\wit d_i(t)$, $\wit a_{ij}(t)$, $\wit c_i(t)$,  and $\wit \beta_{i}(t)$ are almost periodic maps on $\R$;
\item[(a2)] $\wit d_i(t)\geq d_0>0$ for any $t\in\R$,  for any $i$;
\item[(a3)] $\wit a_{ij}(t)$ are all nonnegative maps and $\wit a_{ii}$ is taken to be identically null;
\item[(a4)] $\wit\beta_{i}(t)>0$ for any $t\geq 0$, for any $i$;
\item[(a5)] $\wit c_{i}(t)\geq c_0>0$ for any $t\geq 0$, for any $i$;
\item[(a6)]   $\wit d_i(t)-\sum_{j=1}^n \wit a_{ji}(t)>0$  for any $t\geq 0$, for any $i$.
\end{itemize}
\par
To get a biological meaning of the imposed conditions, the coefficient $\wit a_{ij}(t)$ stands for the migration rate of the population moving from patch $j$ to patch
$i$ at time $t\geq 0$. As for the birth function in each patch, it is given by the delay nonlinear Nicholson term.   Finally, the decreasing rate in patch $i$, given by $\wit d_i(t)$, includes the mortality rate as well as the migrations coming out of patch $i$, so that condition (a6) makes sense, saying that the mortality rate is positive at any time.
\par
From an analytical point of view, condition (a5) is imposed so as to guarantee the uniform boundedness of the terms $y\,e^{-\wit c_i(t)\,y}$ for $y\geq 0$ and $t\in \R$, and condition (a6) is a weak column dominance condition for the matrix of coefficients of the ODEs linear part of system~\eqref{nicholson delay}. This last condition is enough, in this almost periodic setting, to deduce that the null solution of the ODEs linear system is globally exponentially stable. 
Therefore, a direct application of the
variation of constants formula permits to check that system~\eqref{nicholson delay} is dissipative or, in other words, solutions are ultimately bounded
(see Faria et al.~\cite{faos} for more details), and in particular they are defined for all $t\geq 0$.
\par
Note also that the Nicholson system~\eqref{nicholson delay} does not satisfy the {\em quasimonotone condition\/} given in Smith~\cite{smit}, here just called cooperative condition for simplicity (see Section~\ref{sec-preli}), but still solutions starting with a nonnegative initial map, remain nonnegative forever, just by applying the invariance criterion given in Theorem~5.2.1 in~\cite{smit}. Alternatively, one can note that $y_i'(t)\geq -\wit d_i(t)\,y_i(t)$ for $i=1,\ldots,n$, so that by a standard comparison of solutions argument (once more, see~\cite{smit}), we can affirm that $\varphi\geq 0$ implies $y(t,\varphi)\geq 0$ for any $t\geq 0$ and besides, $\varphi\geq 0$ with $\varphi(0)\gg 0$ implies $y(t,\varphi)\gg 0$ for any $t\geq 0$.
\par
Considering the previous properties of the solutions of the population model~\eqref{nicholson delay}, at least two natural approaches to the concept of persistence arise. On the one hand, if at time $t=0$ there are some individuals in every patch, one wonders whether the population will eventually  persist in all the patches, namely, whether  in the long run  the population will overpass a positive lower bound in all patches. On the other hand, if at time $t=0$ there are some individuals at least in one patch, one wants to know whether the population will persist in some patch (possibly a different one). We refer to these situations as  uniform persistence at $0$ and strict persistence at $0$, respectively, and the precise definitions are the following.
\begin{defi}\label{defi persistence Nicholson}
(i) The Nicholson system~\eqref{nicholson delay} is {\em uniformly persistent at $0$} ({\it $u_0$-persistent} for short) if there exists an $m>0$ such that for any initial map $\varphi\geq 0$ with  $\varphi(0)\gg 0$ there exists a time $t_0=t_0(\varphi)$ such that
     \[ y_i(t,\varphi)\geq m \quad \text{for any }\;t\geq t_0 \;\text{ and any }\; i=1,\ldots,n \,.\]
     \par
(ii) The Nicholson system~\eqref{nicholson delay} is {\em strictly persistent at $0$} ({\it $s_0$-persistent} for short) if there exists an $m>0$ such that for any initial map $\varphi\geq 0$ with  $\varphi(0)> 0$ there exists a time $t_0=t_0(\varphi)$ such that at least for one component $i$,
     \[ y_i(t,\varphi)\geq m \quad \text{for any }\;t\geq t_0\,.\]
\end{defi}
Note that both definitions agree with the concept of uniform (strong) $\rho$-per\-sis\-ten\-ce in the terms of Smith and Thieme~\cite{smth} for an adequate choice of the map $\rho:X\to\R_+$ (see also Faria and R\"{o}st~\cite{faro} and Faria et al.~\cite{faos}). Our main purpose is to have a characterization of these two properties in terms of some computable objects related to the system.
\par
Recently, in \cite{obsa} the authors have characterized the properties of uniform persistence and strict persistence at $0$ for families of almost periodic Nicholson systems. If we want to take advantage of their approach, the first thing that we have to do is to include the initial non-autonomous system~\eqref{nicholson delay} into the family of systems over the hull of the vector-valued map determined by all the almost periodic coefficients. Note that we need coefficients of~\eqref{nicholson delay} to be  defined on  $\R$ to easily build the hull $\W$ of the system, and define the continuous translation flow $\R\times\W\to\W$, just denoted by $(t,\w)\mapsto\w{\cdot}t$. Then, for each $\w\in\W$ the corresponding system in the family can be written as
\begin{equation}\label{nicholson delay hull}
y'_i(t)=- d_i(\w{\cdot}t)\,y_i(t) +\des \sum_{j=1}^n  a_{ij}(\w{\cdot}t)\,y_j(t) + \beta_{i}(\w{\cdot}t)\,y_i(t-\tau_{i})\,e^{-c_i(\w{\cdot}t)\,y_i(t-\tau_{i})}\,,
\end{equation}
$i=1,\ldots,n$, for certain continuous nonnegative maps $d_i,\,a_{ij},\,\beta_{i},\,c_i$ defined on $\W$.
\par
We take $X=C([-\tau_1,0])\times \ldots \times C([-\tau_n,0])$ with the usual cone of positive elements, denoted by $X_+$, and the sup-norm.  Then, solutions $y(t,\w,\varphi)$ of systems~\eqref{nicholson delay hull} for $\w\in\W$ with initial values $\varphi\in X_+$ induce a globally defined (see Theorem~\ref{teor obayasanz} (i)) skew-product semiflow~\eqref{skew},  $\R_+\times \W\times X_+\to \W\times X_+$, $(t,\w,\varphi)\mapsto (\w{\cdot}t,y_t(\w,\varphi))$, with the usual notation in delay equations, $y_t(\w,\varphi)_i(s)=y_i(t+s,\w,\varphi)$ for any $s\in [-\tau_i,0]$, for each $i=1,\ldots,n$.  The fact that the set $\W\times X_+$ is invariant for the dynamics follows once more from the criterion given in Theorem~5.2.1 in~\cite{smit}. Besides, this semiflow has a trivial minimal set $K=\W\times \{0\}$, as the null map is a solution of any of the systems over the hull.
\par
In this situation, the properties of uniform persistence and strict persistence at $0$ for the family of systems~\eqref{nicholson delay hull} have the following collective formulation,  directly adapted from  Definitions~3.1 and~5.2 in~\cite{obsa}, respectively.
\begin{defi}\label{defi persistence hull}
(i) The  family of Nicholson systems~\eqref{nicholson delay hull} is {\em uniformly persistent} ({\it u-persistent} for short) if there exists a map $\psi\gg 0$ such that for any $\w\in \W$ and any initial map $\varphi\gg 0$  there exists a time $t_0=t_0(\w,\varphi)$ such that
     $y_t(\w,\varphi)\geq \psi$ for any $t\geq t_0$.\par
(ii) The  family of Nicholson  systems~\eqref{nicholson delay hull} is {\em strictly persistent at $0$} ({\it $s_0$-persistent} for short) if there exists a collection of maps $e_1,\ldots,e_p\in X$, with $e_k>0$ for $k=1,\ldots,p$, such that for any $\w\in \W$ and any initial map $\varphi\geq 0$ with  $\varphi(0)> 0$ there exists a time $t_0=t_0(\w,\varphi)$ such that
$y_t(\w,\varphi)\geq  e_k$ for any $t\geq t_0$, for some $k\in \{1,\ldots,p\}$.
\end{defi}
More precisely, Section 6 in~\cite{obsa} is devoted to the study of these persistence properties for the family of almost periodic Nicholson systems~\eqref{nicholson delay hull}, where the coefficients $\wit c_i(t)$ in the initial system~\eqref{nicholson delay}  have been taken to be identically equal to $1$ just for simplicity.  It is straightforward to check that, under hypothesis (a5), all the results in Section~6 in~\cite{obsa} still apply. For the sake of completeness we include here the following result, whose items are respectively Theorem~6.1 and Theorem~6.2 in~\cite{obsa}.
\begin{teor}\label{teor obayasanz}
Let us consider the Nicholson system~\eqref{nicholson delay} under assumptions  {\rm (a1)-(a6)}. Then:
\begin{itemize}
\item[(i)] Solutions of the family~\eqref{nicholson delay hull}  with initial condition in $X_+$ are ultimately bounded, in the sense that there exists a constant $r>0$ such that for any $\w\in \W$ and any $\varphi\in X_+$,  any component of the vectorial solution satisfies $0\leq y_i(t,\w,\varphi)\leq r$ from some time on. In particular the induced semiflow is globally defined on $\W\times X_+$.
\item[(ii)] The  family of Nicholson systems~\eqref{nicholson delay hull} is uniformly persistent (resp. strictly persistent at $0$) if and only if the linearized family of systems along the null solution, which is given by
\begin{equation}\label{nicholson delay lineal}
z_i'(t)=-d_i(\w{\cdot}t)\,z_i(t) +\des \sum_{j=1}^n  a_{ij}(\w{\cdot}t)\,z_j(t) + \beta_{i}(\w{\cdot}t)\,z_i(t-\tau_{i})\,,
\end{equation}
for $i=1,\ldots,n$, for each $\w\in \W$, is uniformly persistent (resp.~strictly persistent at $0$) in the sense of Definition {\rm\ref{defi persistence hull}}.
\end{itemize}
 \end{teor}
The importance of the first approximation result to check persistence stated in (ii)  lies on the fact that Nicholson systems are not cooperative, as for cooperative systems the result for uniform persistence has already been proved in~\cite{noos7}. Note that the linearized systems along the null solution~\eqref{nicholson delay lineal} are   independent of the coefficients $c_i(\w)$ and they  are cooperative thanks to conditions (a3) and (a4), so that the spectral characterization of uniform persistence and strict persistence at $0$ given in~\cite{obsa} for general cooperative delay linear families directly applies to them. The precise spectral characterization of the persistence properties for the almost periodic Nicholson family~\eqref{nicholson delay hull}  is stated in Theorem~6.3 in~\cite{obsa}.
\par
At this point it is natural to pose some questions:
\begin{itemize}
\item[(Q1)] What is the relation between the definitions of persistence for the initial system~\eqref{nicholson delay} given in Definition~\ref{defi persistence Nicholson}, and the definitions stated in Definition~\ref{defi persistence hull} in a collective way for the family of systems~\eqref{nicholson delay hull}?
\item[(Q2)] Can we give a precise characterization of the persistence properties of system~\eqref{nicholson delay} in terms of some computable items of the system?
\end{itemize}
\par
The purpose of this section is to give an answer to these two questions. In short, we are going to see that things go smoothly for the almost periodic Nicholson systems. We will also determine another class of systems for which things go exactly as in the Nicholson systems. However, as it will be shown in the next section, the transfer of the property of persistence from one particular non-autonomous almost periodic system to the family of systems over the hull, as it occurs in the Nicholson systems, is not to be expected in general: it fails even for one-dimensional linear almost periodic equations.  This supports the convenience of  considering a collective formulation of the properties of persistence on the family of systems over the hull, as otherwise, the property may hold for some systems of the family and may not hold for others, which is not desirable in applications to real world models.
\begin{teor}\label{teor Nicholson pers}
Let us consider the almost periodic Nicholson system~\eqref{nicholson delay} under assumptions {\rm (a1)-(a6)} and the family of systems over the hull~\eqref{nicholson delay hull}. Then,
system~\eqref{nicholson delay} is uniformly persistent at $0$ (resp.~strictly  persistent at $0$)  if and only if the family of systems~\eqref{nicholson delay hull} is uniformly persistent (resp.~strictly  persistent at $0$).
\end{teor}
\begin{proof}
In both cases, the easy implication is the one transferring the persistence property from the family to the initial system. In the case of u-persistence, for the map $\psi\gg 0$ given in Definition~\ref{defi persistence hull} (i), define $m=\min\{\psi_1(0),\ldots,\psi_n(0)\}>0$. Now, in order to check the $u_0$-persistence of the initial system, recall that it is one of the systems in the family over the hull: let it be the system for $\w_0\in \W$, and let us keep this notation throughout  the whole proof. Now, fixed any initial map $\varphi\geq 0$ with $\varphi(0)\gg 0$, note that the solution $y(t,\varphi)=y(t,\w_0,\varphi)\gg 0$ for any $t\geq 0$, so that for $\tau_0=\max\{\tau_1,\ldots,\tau_n\}$ it holds that $y_{\tau_0}(\w_0,\varphi)\gg 0$. Therefore, by the u-persistence of the family, for $\w_0{\cdot}\tau_0$ and $y_{\tau_0}(\w_0,\varphi)\gg 0$ there exists a time $t_0=t_0(\w_0,\tau_0,\varphi)$ such that $y_t(\w_0{\cdot}\tau_0,y_{\tau_0}(\w_0,\varphi))\geq \psi$ for any $t\geq t_0$. By the cocycle property, this means that $y_{t+\tau_0}(\w_0,\varphi)\geq \psi$ for any $t\geq t_0$, and therefore, $y_i(t,\varphi)\geq m$ for any $t\geq t_0+\tau_0$ and any $i=1,\ldots,n$, and we are done.
\par
As for the case of $s_0$-persistence, for each of the maps $e_k>0$ given in Definition~\ref{defi persistence hull} (ii), there is at least one component $i=i(k)$ such that $(e_k)_i>0$, so that there exists at least a $s_k\in [-\tau_i,0]$ with $(e_k)_i(s_k)>0$. Now, define $m=\min\{(e_1)_{i(1)}(s_1),\ldots,(e_p)_{i(p)}(s_p)\}>0$. Then, given  $\varphi\geq 0$ with $\varphi(0)> 0$, by the $s_0$-persistence of the family there exists a time $t_0=t_0(\w_0,\varphi)$ such that $y_t(\w_0,\varphi)\geq e_k$ for $t\geq t_0$, for some $k\in\{1,\ldots,p\}$. In particular, for the component $i=i(k)$ previously defined,  $y_t(\w_0,\varphi)_{i}(s_k)=y_i(t+s_k,\varphi)\geq (e_k)_{i}(s_k)\geq m$ for any $t\geq t_0$, so that $y_i(t,\varphi)\geq m$ for any  $t\geq t_0$, as we wanted.
\par
For the converse implication in the case of u-persistence the arguments are more subtle, and we make use of the general theory of monotone and concave $\mathcal{C}^1$ skew-product semiflows developed by N\'{u}\~{n}ez et al.~in~\cite{nuos4}.  To begin with, taking condition (a5) into consideration we note that for any $y\geq 0$, $\w\in\W$ and $i=1,\ldots,n$, $y\,e^{-c_i(\w)\,y}\leq y\,e^{-c_0\,y}$. Then, we define the nondecreasing, bounded and concave map $h:[0,\infty)\to [0,\infty)$ of class $\mathcal C^1$,
\[
h(y)= \left\{ \begin{array}{ll}
y\,e^{-c_0\,y} &\text{if }\; y\in[0,1/c_0]\,,  \\
\frac{1}{c_0}\,e^{-1} &\text{if }\; y\in[1/c_0,\infty) \,,
\end{array} \right.
\]
we look at the family of cooperative and concave delay nonlinear systems given for each $\w\in \W$ by
\begin{equation}\label{comparar}
z'_i(t)=- d_i(\w{\cdot}t)\,z_i(t) +\des \sum_{j=1}^n  a_{ij}(\w{\cdot}t)\,z_j(t) + \beta_{i}(\w{\cdot}t)\,h(z_i(t-\tau_{i}))\,,
\end{equation}
for $i=1,\ldots,n$, where the coefficients are just those of~\eqref{nicholson delay hull}, and consider the induced skew-product semiflow $\bar \tau:\R^+\times\W\times X_+\to \W\times X_+$, $(t,\w,\varphi)\mapsto (\w{\cdot}t,z_t(\w,\varphi))$, where $z(t,\w,\varphi)$ is the solution of system~\eqref{comparar} with initial value $\varphi$. As the nonlinear terms are uniformly bounded, the same argument as that in the proof of Theorem~6.1 in~\cite{obsa} implies that solutions of~\eqref{comparar} are ultimately bounded and in particular $\bar\tau$ is globally defined. Besides, since the systems are cooperative and concave, the semiflow is monotone and concave, and it is also $\mathcal C^1$.
\par
Now, assuming that the property of $u_0$-persistence in Definition~\ref{defi persistence Nicholson} (i) holds for system~\eqref{nicholson delay}, take an initial map $\varphi_0\geq 0$ with $\varphi_0(0)\gg 0$ and take $t_0=t_0(\varphi_0)$ such that $y_i(t,\varphi_0)=y_i(t,\w_0,\varphi_0)\geq m$ for any $t\geq t_0$ and any $i=1,\ldots,n$. Then, as systems~\eqref{comparar} are cooperative, we can apply a standard argument of comparison of solutions to state that for $t\geq t_0$, $m\leq y_i(t,\w_0,\varphi_0)\leq z_i(t,\w_0,\varphi_0)$.
\par
At this point, we can build the omega-limit set $\mathcal{O}(\w_0,\varphi_0)$ of the pair $(\w_0,\varphi_0)$ for the semiflow $\bar \tau$, which contains a minimal set $K$ which necessarily lies on the zone $\W\times\{\varphi\in X_+\mid \varphi\geq \bar m\}$, for the map $\bar m\in X$ whose components are identically equal to $m$. In other words, there is a strongly positive minimal set for $\bar \tau$. Then, Theorem~3.8 in~\cite{nuos4} applied to the $\mathcal{C}^1$ monotone and concave skew-product semiflow $\bar \tau$ asserts that the dynamics suits one the following cases: the so-called case A1 when $K$ is the unique minimal set strongly above $0$, or case A2 when there are infinitely many minimal sets strongly above $0$.
\par
If we can discard case A2, we are done, as in case A1 the unique minimal set is a hyperbolic copy of the base, that is, $K=\{(\w,c(\w))\mid \w\in\W\}$ for certain continuous map $c:\W\to X_+$, which exponentially attracts  any trajectory starting inside the interior of the positive cone. Therefore, it is immediate that the semiflow $\bar \tau$ is  u-persistent in the interior of the positive cone according to Definition 3.1 in~\cite{obsa} or, in other words, the family~\eqref{comparar} is u-persistent. Now, for $\bar \tau$ regular monotone and concave, with $0$ being a trajectory, it holds that
\begin{equation}\label{mayorante lineal}
z_t(\w,\varphi) \leq D_\varphi z_t(\w,0)\,\varphi\,, \quad \text{for any}\;\w\in\W\,,\; \varphi\geq 0 \;\text{and}\; t\geq 0\,,
\end{equation}
and it is well-known that $z(t)=(D_\varphi z_t(\w,0)\,\varphi)(0)$ provide the solutions of the linearized family  along $0$ of the family~\eqref{comparar}, which by construction coincides with~\eqref{nicholson delay lineal},  the family of linearized Nicholson systems along $0$, which then turn out to be  u-persistent. In this case, to finish, we can apply Theorem~\ref{teor obayasanz} (ii)  to conclude the u-persistence for the Nicholson family~\eqref{nicholson delay hull}.
\par
Finally, we discard case A2. Argue for contradiction and assume that case A2 holds for $\bar\tau$. Then, according to the proof of Theorem~3.8 in~\cite{nuos4} we can consider the family of strongly positive minimal sets  $K_s=\mathcal{O}(\w_0,s\,\bar\varphi)$ for $s\in (0,1]$ for a fixed $\bar\varphi\gg 0$ with $(\w_0,\bar\varphi)\in K$, which  must satisfy property (vi) in the statement of case A2:
if $(\w,\psi)\in\W\times X_+$  is such that for any $s\in(0,1]$
there exists $(\w,\varphi_s)\in K_s$ with $\psi\le \varphi_s$, then $\psi\not\gg 0$. Nevertheless, by the  $u_0$-persistence, as done before, we have that  $K_s\subset \W\times\{\varphi\in X_+\mid \varphi\geq \bar m\}$
for any $s\in (0,1]$, and we get a contradiction just by taking $\psi =\bar m/2\gg 0$. We are finished.
\par
It remains to deal with the $s_0$-persistence property of the family, assuming the $s_0$-persistence property of the initial system. Once more this is quite delicate and the proof follows the line of ideas used in~\cite{obsa}, in what refers to a rearrangement of the family of systems in view of the linearized family. Recall here that a square matrix $A=[a_{ij}]$ is {\em reducible} if there is a simultaneous permutation of rows and columns that brings $A$ to the form
  $$ \left[\begin{array}{cc}
   A_{11}&0\\ A_{21}&A_{22}\end{array}\right],$$
  with $A_{11}$ and $A_{22}$   square matrices; and it is {\em irreducible} if it is not reducible. Equivalently, for $n>1$, $A$ is irreducible if for any nonempty proper subset
$I\subset\{1,\ldots,n\}$ there are $i\in I$ and $j\in
\{1,\ldots,n\}\setminus I$ such that $a_{ij}\not=0$.
\par
More precisely, as stated in Theorem~6.3 in~\cite{obsa}, for each $\w\in\W$ we can look at the linearized system along the null solution~\eqref{nicholson delay lineal} and assume without loss of generality that the constant matrix $\bar A=[\bar a_{ij}]$ defined as
\begin{equation}\label{a}
\bar a_{ij}=\sup_{\w\in\W} a_{ij}(\w)\;\text{ for }\,i\not= j\,,\quad \text{and }\, \bar a_{ii}=0\,
\end{equation}
has a block lower triangular structure
\begin{equation}\label{triangular}
\left[\begin{array}{cccc}
\bar A_{11} & 0  &\ldots & 0 \\
\bar A_{21} & \bar A_{22} & \ldots& 0 \\
\vdots & \vdots  &\ddots & \vdots \\
\bar A_{k1} & \bar A_{k2} & \ldots& \bar A_{kk}
\end{array}\right]\,,
\end{equation}
with irreducible diagonal blocks $\bar A_{jj}$ of dimension $n_j$ for $j=1,\ldots,k$ ($n_1+\cdots+n_k=n$).
To simplify the notation, we arrange the set of delays by blocks by denoting $\{\tau_1,\ldots,\tau_n\}=\{\tau^1_1,\ldots,\tau^1_{n_1},\ldots,\tau^k_1,\ldots,\tau^k_{n_k}\}$ and we write $X=X^{(1)}\times\ldots\times X^{(k)}$ for
\begin{equation}\label{xj}
X^{(j)}= C([-\tau^j_{1},0])\times \ldots \times C([-\tau^j_{n_j},0])\,,\quad j=1,\ldots,k\,.
\end{equation}
For each $j=1,\ldots,k$ let $L_j$ be the linear skew-product semiflow induced on the product space $\W\times X^{(j)}$ by the solutions of the $n_j$-dimensional delay linear systems corresponding to the $j\,$th diagonal block in~\eqref{nicholson delay lineal},
\begin{equation}\label{linearized j}
z_i'(t)=-d_i(\w{\cdot}t)\,z_i(t) +\des \sum_{l\in I_j} a_{il}(\w{\cdot}t)\,z_l(t) + \beta_{i}(\w{\cdot}t)\,z_i(t-\tau_{i})\,,\quad t\geq 0\,,
\end{equation}
for $i\in I_j$, for each $\w\in \W$, where $I_j$ is  the set formed by
the $n_j$ indexes corresponding to the rows of the block $\bar A_{jj}$. Then, $L_j$ admits a continuous separation (of type II) and its principal spectrum is just given by the upper Lyapunov exponent $\lambda_j$ of the minimal set $K^j=\W\times \{0\}\subset \W\times X^{(j)}$. Besides, Theorem~6.3 in~\cite{obsa} gives a precise characterization of the properties of u-persistence and $s_0$-persistence for the Nicholson family~\eqref{nicholson delay hull} in terms of the positivity of a certain set of these exponents $\lambda_j$ in each case. Now we distinguish two cases.
\par\smallskip
\noindent
\textbf{(C1)}: $k=1$, that is, the matrix $\bar A$ is irreducible. In this case, starting with a positive component of the solution, we are going to raise the other ones, so as to actually obtain $u_0$-persistence for system~\eqref{nicholson delay}. We remark that a similar argument has been used in the proof of Theorem~5.4 in~\cite{noos7}. More precisely, given $\varphi\geq 0$ with $\varphi(0)>0$ there exists a $t_0=t_0(\varphi)$ and there exists a component $i_1$ such that $y_{i_1}(t,\varphi)=y_{i_1}(t,\w_0,\varphi)\geq m$ for any $t\geq t_0$, for the constant $m>0$ given in Definition~\ref{defi persistence Nicholson} (ii). Now, as $\bar A$ is irreducible, there exists an index $i_2\in \{ 1,\ldots,n\}\setminus\{i_1\}$ such that $\bar a_{i_2i_1}>0$. As  $y_{i_2}'(t,\w_0,\varphi)\geq -d_{i_2}(\w_0{\cdot}t)\,y_{i_2}(t,\w_0,\varphi)+a_{i_2i_1}(\w_0{\cdot}t)\,m$ for $t\geq t_0$, we consider the scalar family of ODEs for $\w\in\W$,
\begin{equation}\label{escalar}
h'(t)= -d_{i_2}(\w{\cdot}t)\,h(t)+a_{i_2i_1}(\w{\cdot}t)\,m \,,
\end{equation}
written for short as $h'(t)=F(\w{\cdot}t,h(t))$, for which the null map is a lower solution because $F(\w,0)\geq 0$ for any $\w\in\W$. Besides, since $\bar a_{i_2i_1}>0$, there exists an $\w^*\in\W$ such that $F(\w^*,0)=a_{i_2i_1}(\w^*)\,m> 0$. In this situation $0$ is a strong sub-equilibrium (see Lemma~3.15 in~\cite{noob1}, which applies to ODEs). As a consequence, there exist $t_{i_2}>0$ and $m_{i_2}>0$ such that, if $h(t,\w,0)$ is the solution of~\eqref{escalar}
with initial value $0$, then $h(t,\w,0)>m_{i_2}$ for any $t\geq t_{i_2}$ and any $\w\in \W$. Therefore, a standard argument of comparison of solutions leads to the fact that, for any $t\geq t_0+t_{i_2}$,
\[
y_{i_2}(t,\w_0,\varphi)\geq h(t-t_0,\w_0{\cdot}t_0,y_{i_2}(t_0,\w_0,\varphi))\geq h(t-t_0,\w_0{\cdot}t_0,0)>m_{i_2}\,.
\]
\par
Now, if there are any more components, the process is just the same. We just give a sketch for the next step. By the irreducible character of $\bar A$, there exist indexes $i_3\in \{ 1,\ldots,n\}\setminus\{i_1,i_2\}$ and $i\in \{i_1,i_2\}$ such that $\bar a_{i_3i}>0$, and then note that $y_{i_3}'(t,\w_0,\varphi)\geq -d_{i_3}(\w_0{\cdot}t)\,y_{i_3}(t,\w_0,\varphi)+a_{i_3i}(\w_0{\cdot}t)\,m_i$ for $t\geq t_0+t_{i_2}$, for the constant $m_i$ given by  $m$ if $i=i_1$ and by $m_{i_2}$ if $i=i_2$. The same argument as before leads to the existence of some $t_{i_3}>0$ and $m_{i_3}>0$ such that $y_{i_3}(t,\w_0,\varphi)\geq m_{i_3}$ for any $t\geq t_0+t_{i_2}+t_{i_3}$. Iterating the process we finally obtain that, taking $m_0=\min\{m,m_{i_2},\ldots,m_{i_n}\}$,  $y_{i}(t,\w_0,\varphi)\geq m_0$ for any $t\geq t_0+t_{i_2}+\ldots+t_{i_n}$ and any $i=1,\ldots,n$.
\par
Note that the constant $m_0$ just defined depends on the component $i_1=i_1(\varphi)$ we started with. As there are just $n$ different components with which the process can start, depending on the initial map $\varphi$, and the process in each case exclusively depends on the irreducible structure of the constant matrix $\bar A$, we can conclude that system~\eqref{nicholson delay} is $u_0$-persistent. As we already know,  this property extends as u-persistence to the whole  family~\eqref{nicholson delay hull}, $\w\in\W$, and we can apply Theorem~6.3 in~\cite{obsa}, which in the case $k=1$ says that the upper Lyapunov exponent $\lambda>0$, and the family is also $s_0$-persistent. We are done with this case.
\par\smallskip
\noindent
\textbf{(C2)}: $k>1$, that is, the matrix $\bar A$ is reducible and it has the block lower triangular structure~\eqref{triangular}. In this case Theorem~6.3 in~\cite{obsa} asserts that the Nicholson family~\eqref{nicholson delay hull} is $s_0$-persistent if and only if $\lambda_j>0$ for any $j\in J$ for the set of indexes
\[
J=\{j\in\{1,\ldots,k\} \,\mid\, \bar A_{ij}=0 \text{ for any } i\not= j\}.
\]
Thus, we fix $j\in J$ and we consider the $n_j$-dimensional Nicholson-type system
\begin{equation}\label{nicholson j}
y_i'(t)=-\wit d_i(t)\,y_i(t) +\des \sum_{l\in I_j} \wit a_{il}(t)\,y_l(t) + \wit\beta_{i}(t)\,y_i(t-\tau_{i})\,e^{-\wit c_i(t)\,y_i(t-\tau_{i})}\,,
\end{equation}
for $i\in I_j$, which is included for $\w=\w_0$ in the family of systems for  $\w\in\W$,
\begin{equation*}
y_i'(t)=-d_i(\w{\cdot}t)\,y_i(t) +\des \sum_{l\in I_j}  a_{il}(\w{\cdot}t)\,y_l(t) + \beta_{i}(\w{\cdot}t)\,y_i(t-\tau_{i})\,e^{-c_i(\w{\cdot}t)\,y_i(t-\tau_{i})}\,,
\end{equation*}
for $i\in I_j$, with linearized family along $0$ given by~\eqref{linearized j} and associated constant matrix $\bar A_{jj}$, which is irreducible. If system~\eqref{nicholson j} is $s_0$-persistent, we can apply to it the result in case (C1) to get that the upper Lyapunov exponent $\lambda_j>0$.
\par
So, to finish, let us check that for each $j\in J$ system~\eqref{nicholson j} is $s_0$-persistent. For that, take $\bar\varphi^j\in X^{(j)}_+$ with $\bar\varphi^j(0)>0$, and build a map $\varphi\in X_+=X^{(1)}_+\times\ldots\times X^{(k)}_+$, $\varphi=(\varphi^1,\ldots,\varphi^k)$ such that $\varphi^j=\bar\varphi^j$ and $\varphi^i=0$ for $i\not=j$, which satisfies $\varphi\geq 0$ and $\varphi(0)> 0$. The $s_0$-persistence of the initial system~\eqref{nicholson delay} says that there exists a $t_0=t_0(\varphi)$ such that for some component $i_0$, $y_{i_0}(t,\varphi)\geq m$ for $t\geq t_0$. Now, by the structure of the system noting that $j\in J$, and the structure of the initial map $\varphi$, it is easy to check that, writing the solution by blocks $y(t,\varphi)=(y^1(t,\varphi),\ldots,y^k(t,\varphi))$, it is  $y^i(t,\varphi)=0$ for $i\not=j$, whereas $y^j(t,\varphi)$ coincides with the solution of system~\eqref{nicholson j} with initial condition $\varphi^j=\bar\varphi^j$. Therefore, necessarily $i_0\in I_j$ and we are done. The proof is finished.
\end{proof}
Once we have given a satisfactory answer to question (Q1), we now present an answer to question (Q2).
\begin{teor}\label{teor Nicholson system}
Let us consider the almost periodic Nicholson system~\eqref{nicholson delay} under assumptions {\rm (a1)-(a6)}, and let us assume without loss of generality that the constant matrix $\bar A=[\bar a_{ij}]$ defined as
\[
\bar a_{ij}=\sup_{t\in \R} \wit a_{ij}(t)\;\text{ for }\,i\not= j\,,\quad \text{and }\, \bar a_{ii}=0\,
\]
has a block lower triangular structure as in~\eqref{triangular}
with irreducible diagonal blocks $\bar A_{jj}$ of dimension $n_j$ for $j=1,\ldots,k$ $(n_1+\cdots+n_k=n)$. For each $j=1,\ldots,k$ let us consider the $n_j$-dimensional almost periodic linear delay system
\begin{equation}\label{nicholson lineal j}
z_i'(t)=-\wit d_i(t)\,z_i(t) +\des \sum_{l\in I_j} \wit a_{il}(t)\,z_l(t) + \wit\beta_{i}(t)\,z_i(t-\tau_{i})\,,\quad t\geq 0\,,
\end{equation}
for $i\in I_j$, the set of indexes corresponding to the rows of the block $\bar A_{jj}$, and let $z^j(t,\bar 1)$ be the solution with initial map $\bar 1$, the map with all components identically equal to $1$ in the space $X^{(j)}$ defined in~\eqref{xj}. Then, let $\wit\lambda_j$ be defined as
\begin{equation*}
\wit\lambda_j=\lim_{t\to \infty} \frac{\log\|z_t^j(\bar 1)\|}{t}\,.
\end{equation*}
\par
Finally, let us consider two sets of indexes associated to the structure of the linear part of the system: if $k=1$, i.e., if the matrix $\bar A$ is irreducible, let $I=J=\{1\}$; else, let
\begin{align*}
I&=\{j\in\{1,\ldots,k\} \,\mid\, \bar A_{ji}=0 \text{ for any } i\not= j\},\\
J&=\{j\in\{1,\ldots,k\} \,\mid\, \bar A_{ij}=0 \text{ for any } i\not= j\},
\end{align*}
that is, $I$ is composed by the indexes $j$ such that any off-diagonal block in the row of $\bar A_{jj}$ is null, whereas $J$ contains those indexes $j$ such that any off-diagonal block in the column of $\bar A_{jj}$ is null.
Then:
\begin{itemize}
 \item[(i)] The almost periodic Nicholson system~\eqref{nicholson delay} is uniformly persistent at $0$ if and only if $\wit\lambda_j>0$ for any $j\in I$.
  \item[(ii)] The almost periodic Nicholson system~\eqref{nicholson delay} is strictly persistent at $0$ if and only if $\wit\lambda_j>0$ for any $j\in J$.
 \end{itemize}
\end{teor}
\begin{proof}
First of all, recall that if the matrix $\bar A$ does not have the required structure, we just need to permute the variables in order to obtain it. Also, note that when the Nicholson system is included in the family of systems~\eqref{nicholson delay hull} over the hull  $\W$, the matrix $\bar A$ defined under the same name in~\eqref{a} coincides with the matrix $\bar A$ here defined, because of the hull construction.
\par
Now, as stated in Theorem~\ref{teor Nicholson pers}, the properties of $u_0$-persistence and $s_0$-persistence of system~\eqref{nicholson delay} are equivalent respectively to the properties of u-persistence and $s_0$-persistence for the family of systems~\eqref{nicholson delay hull}, and the last properties are completely characterized in Theorem~6.3 in~\cite{obsa}, which is a parallel result to the above one just given in terms of the upper Lyapunov exponents $\lambda_j$ of the trivial minimal set $K^j=\W\times\{0\}$ for the linear skew-product semiflow $L_j$ induced on $\W\times X^{(j)}$ (see~\eqref{xj}) by the solutions of the $n_j$-dimensional linear delay family~\eqref{linearized j}, for each $j=1,\ldots,k$.
\par
At this point, it remains to check that the number $\wit\lambda_j$ coincides with $\lambda_j$ for each $j=1,\ldots,k$. The thing is that, as already commented before, thanks to the irreducible character of the diagonal block $\bar A_{jj}$, the linear skew-product semiflow $L_j(t,\w,\varphi)=(\w{\cdot}t,\Phi_j(t,\w)\,\varphi)$ admits a continuous separation (of type II), which roughly speaking means that there is an invariant one-dimensional subbundle dominating the dynamics of $L_j$ in the long run.
More precisely, if $X^{(j)}=X_1^{(j)}(\w)\oplus X_2^{(j)}(\w)$ for $\w\in\W$ is the decomposition given by the continuous separation of $L_j$, with $X_1^{(j)}(\w)=\spa\{ v^j(\w)\}$, for a continuous map $v^j:\W\to X^{(j)}$ such that   $v^j(\w)\gg 0$ and $\|v^j(\w)\|=1$ for any $\w\in \W$, then $\Phi_j(t,\w)\,v^j(\w)=c_j(t,\w)\,v^j(\w{\cdot}t)$, and the positive coefficients $c_j(t,\w)$, which can be defined for all $t\in\R$ and $\w\in\W$, satisfy the linear cocycle property $c_j(t+s,\w)=c_j(t,\w{\cdot}s)\,c_j(s,\w)$
for any $t,s\in \R$ and any $\w\in \W$ (the reader is referred to~\cite{noos7} or~\cite{obsa} for more details). That is, the one-dimensional linear
skew-product flow given by the scalar cocycle $c_j(t,\w)$ can be seen as a flow extension of the restriction
of the linear semiflow $L_j$ to the leading one-dimensional
subbundle,  turning the problem into the setting of the spectral theory for one-dimensional linear skew-product flows, which has been studied in Sacker and Sell~\cite{sase}.
\par
Note also that the almost periodicity of the coefficients implies that the flow in $\W$ is uniquely ergodic, so that the Sacker-Sell spectrum of the previous one-dimensional linear skew-product flow (which is called  the principal spectrum by definition) reduces to a singleton, namely,  $\{\lambda_j\}$,  for  the upper Lyapunov exponent
$\lambda_j=\sup_{\w\in\W}\lambda_j(\w)$, where the Lyapunov exponent for each $\w\in\W$ is defined
as
\[
\lambda_j(\w)= \limsup_{t\to\infty} \frac{\log\|\Phi_j(t,\w)\|}{t}\,.
\]
\par
Now, Remark~2 in~\cite{sase} says that the result given in Theorem~7 for almost periodic linear ODEs does extend to the case of differentiable linear skew-product flows on vector bundles, provided that the base flow is minimal and uniquely ergodic. Therefore, we can apply Theorem~7 in~\cite{sase} to the one-dimensional invariant subbundle determined by the continuous separation, so that the upper Lyapunov exponent $\lambda_j$ can be calculated along the trajectories in the one-dimensional subbundle. More precisely, for all $\w\in\W$ there exists the limit
\[
\lambda_j=\lim_{t\to\infty} \frac{\log\|\Phi_j(t,\w)\,v^j(\w)\|}{t}= \lim_{t\to\infty} \frac{\log\|c_j(t,\w)\,v^j(\w{\cdot}t)\|}{t} =  \lim_{t\to\infty} \frac{\log c_j(t,\w)}{t}\,,
\]
and the last limit has been shown in~\cite{noos7} to give the value $\lambda_j(\w)$, so that the value of the upper Lyapunov exponent $\lambda_j$ is attained at any $\w\in\W$.
\par
In particular, for $\w_0$ giving the initial system~\eqref{nicholson delay}, one can calculate $\lambda_j=\lambda_j(\w_0)$.
To finish, it is well-known that for any fixed $\varphi_0\in X^{(j)}$ with $\varphi_0\gg 0$, the norm of the
differential operators $\|\Phi_j(t,\w)\|$ can be controlled by
$\|\Phi_j(t,\w)\,\varphi_0\|$, namely, there exists an $l=l(\varphi_0)>0$ such that
$\|\Phi_j(t,\w)\|\leq l\,\|\Phi_j(t,\w)\,\varphi_0\|$ for any $t>0$ and
$\w\in \W$. From here, taking  $\varphi_0 = \bar 1$ just for the sake of simplicity, and noting that $\Phi_j(t,\w_0)\,\bar 1 = z_t^j(\bar 1)$, we conclude that $\lambda_j=\wit\lambda_j$, as we wanted. The proof is finished.
\end{proof}
\begin{nota} The existence of a continuous separation for the linear semiflows $L_j$ is crucial to reduce the problem to the setting of $1$-dimensional dynamics. Also the fact that the flow on $\W$ is minimal and uniquely ergodic is crucial in two respects: first, to guarantee that the principal spectral intervals reduce to singletons, and second, to permit the calculus of the upper Lyapunov exponent as the Lyapunov exponent of any point $\w\in\W$.
\end{nota}
The advantage of the previous result is that, given a concrete Nicholson system, one can easily compute the matrix $\bar A$ and permute the variables so as to get the required block triangular structure. After that, to estimate the numbers $\wit\lambda_j$, one has to numerically solve the linear delay systems~\eqref{nicholson lineal j}, corresponding to each of the diagonal blocks in $\bar A$, starting with a strongly positive map, which in the statement has been taken to be $\bar 1$, but it might be any other. This can be done in many different ways. The reader is referred to Breda and Van Vleck~\cite{brva} for a general approach and to Calzada et al.~\cite{calz} for a recent approach in the quasi-periodic case, taking advantage of the presence of a continuous separation for the linear skew-product semiflows $L_j$ defined in the previous proof.
\par
To finish this section, we remark that for systems with a similar structure things regarding persistence go as in the almost periodic Nicholson systems.  For instance, the same results can be stated for useful almost periodic population models  which are written as
\begin{equation*}
y'_i(t)=- \wit d_i(t)\,y_i(t) +\des \sum_{j=1}^n \wit a_{ij}(t)\,y_j(t) + \wit\beta_{i}(t)\,h_i(y_i(t-\tau_{i}))\,,
\end{equation*}
for $i=1,\ldots,n$, with similar hypotheses on the linear part to the ones imposed in the Nicholson systems, and where the nonlinearities are of the form
\[
h_i(y)=\frac{y}{1+\wit c_i(t)\,y^\alpha}\quad (\alpha \geq 1)\,,\quad y\in \R_+.
\]
For instance, see the scalar model for the process of hematopoiesis for a population of mature circulating cells studied in Mackey and Glass~\cite{magl}.
\par
Here we collect some analytical features of all these systems which make things go nicely in what refers to questions (Q1) and (Q2).
\begin{enumerate}
\item The almost periodicity of the coefficients, which produces a minimal and uniquely ergodic hull.
\item The ODE linear part of the system is cooperative. It is also uniformly asymptotically stable and the nonlinearities are bounded, which makes the system dissipative.
\item The nonlinear terms $h_i(y)$ are, apart from bounded, sublinear maps, and they are increasing in a right neighborhood of $0$,  so that the induced skew-product semiflow is monotone and sublinear in a region of the phase space.
\item Thanks to (2) and (3), the persistence properties of  the family of systems over the hull can be studied through the linearized family along the null solution.
\end{enumerate}
Note that for $\alpha=1$, the map in the family of nonlinearities is just given by  $h_i(y)=\frac{y}{1+\wit c_i(t)\,y}$ which is always increasing and concave, so that the system is in this case dissipative, cooperative and concave, and Theorem~\ref{teor sublineal} in Section~\ref{sec-ejemplo} applies to it. This kind of nonlinearities have been used in epidemic models with positive feedback; for instance, see Capasso~\cite{capa} and Zhao~\cite{zhaox}.
\section{Uniform persistence in cooperative linear\//\/sublinear models: an individual or a collective property? }\label{sec-ejemplo}
In this section, we first provide a precise example in which the property of uniform persistence is not transferred from one particular non-autonomous almost periodic equation to the family of equations over the hull. In other words, we can affirm that uniform persistence is not a robust property in almost periodic equations. Note that neither is robust the property of strict persistence, since uniform and strict persistence are equivalent properties in the general case of linear monotone skew-product semiflows with a continuous separation of classical type (see~\cite{obsa}). Besides, in the cooperative linear or sublinear setting, if the property of uniform persistence is not inherited by the family, it happens in a strong way, meaning that there might be just a few systems in the family which are uniformly persistent. Thinking of applications, in models of real world processes given by cooperative and linear\//\/sublinear systems of ODEs or delay FDEs, it is highly improbable that we can experimentally or  numerically detect uniform persistence under these circumstances.
\par
As a consequence, there is a general need for definitions of persistence given globally for the family of systems over the hull of a particular non-autonomous system with a recurrent behaviour in time. This is the collective approach that has been taken in~\cite{noos7} and~\cite{obsa}.
This supports the coherence of the results in the previous section, as what happens in Nicholson systems regarding persistence cannot at all be given for granted. In connection with this, another general class of systems inside the class of globally cooperative and sublinear systems is determined, for which the individual uniform persistence implies the collective uniform persistence.
\par
Following this outline, first of all we characterize the property of uniform persistence in the case of a scalar linear ODE.
\begin{prop}\label{prop-linear ap}
Given a continuous function $a:\R\to\R$, let us consider the scalar linear equation
\begin{equation*}
y'(t)=a(t)\,y(t)\,,\quad t\in\R\,,
\end{equation*}
and for each $y_0\in\R$ let us denote by $y(t,y_0)$ the solution such that $y(0,y_0)=y_0$.
Then, the following conditions are equivalent:
\begin{itemize}
\item[(i)]  The equation is uniformly persistent, in the sense that there exists an $m>0$ such that for any $y_0>0$ there exists a $t_0=t_0(y_0)$ such that $y(t,y_0)\geq m$ for any $t\geq t_0$.
\item[(ii)] $\displaystyle\lim_{t\to\infty} \int_0^t a(s)\,ds=\infty$.
\end{itemize}
\end{prop}
\begin{proof}
(i)$\Rightarrow$(ii) To see that the limit is infinity, take any $M>0$. Then, we can take $y_0>0$ small enough so that $\log(m/y_0)>M$. Now, associated to $y_0$ there exists a $t_0=t_0(y_0)$ such that $y(t,y_0)\geq m$ for any $t\geq t_0$. Now, as in this scalar linear case
\[
y(t,y_0)=y_0\,e^{\int_0^t a(s)\,ds}\,,\quad t\in\R\,,
\]
it follows immediately that $\int_0^t a(s)\,ds\geq M$ for any $t\geq t_0$, and we are done.
\par
(ii)$\Rightarrow$(i) In this situation, all solutions with positive initial data go to $\infty$ as $t\to\infty$, so that the definition of uniform persistence holds for any value of $m>0$.
\end{proof}
We now provide the announced example. It is based on a previous example given by Conley and Miller~\cite{comi}, although examples of the same nature date back to the end of the nineteenth century in the work by Poincar\'{e}: for instance, see~\cite{poin}.
\begin{eje}\label{ejemplo}
Let $f(t)$ be the map constructed  in Conley and Miller~\cite{comi} with the following properties: \begin{itemize}
\item[i)] $f:\R\to \R$ is almost periodic;
\item[ii)] $\displaystyle\lim_{t\to\infty} \int_0^t f(s)\,ds=\infty$;
\item[iii)] $f$ has zero mean value, that is, $\displaystyle\lim_{t\to\infty} \frac{1}{t}\,\int_0^t f(s)\,ds=0$.
\end{itemize}
\par
In this situation, on the one hand one looks at the equation $y'(t)=f(t)\,y(t)$, which satisfies condition (ii) in Proposition~\ref{prop-linear ap}, so that it is uniformly persistent; and, on the other hand, we consider the family of linear scalar almost periodic equations over the hull  $\W$ of $f$, that is,
\begin{equation}\label{linear ap hull}
y'(t)= g(t)\,y(t)\,,\quad t\in\R\,, \quad\text{ for each } \; g\in \W\,,
\end{equation}
which is often written as $y'(t)= F(\w{\cdot}t)\,y(t)$, $t\in\R$, for each $\w\in \W$, for the continuous map $F$ on $\W$ given by $F(\w)=\w(0)$.
\par
Now, once more according to Definition 3.1  in~\cite{obsa}, we say that the family of equations over the hull~\eqref{linear ap hull} is {\em uniformly persistent\/} if there exists an $m>0$ such that for any $g\in\W$ and any $y_0>0$ there exists a time $t_0=t_0(g,y_0)$ such that $y(t,g,y_0)\geq m$ for any $t\geq t_0$, where $y(t,g,y_0)$ is the solution of the equation given by $g$ with initial value $y_0$ at time $t=0$.
\par
Thus, the condition of uniform persistence for the family of equations on the hull needs condition  (ii) in Proposition~\ref{prop-linear ap} to be satisfied for any $g\in \W$. However, this is not the case in this concrete example. On the one hand, the set of maps $g$ in $\W$ for which the corresponding equation is not uniformly persistent is of full measure. This is a corollary of Theorem~1 in Shneiberg~\cite{shne}: under the zero mean value assumption on $f$, for almost every  $g$ in $\W$ there exists a sequence $t_n\to \infty$ such that  $\int_0^{t_n} g(s)\,ds = 0$ for every $n\geq 1$. Note that these $g$ are among the set of so-called (Poincar\'{e}) {\it recurrent\/} points at $\infty$, meaning that there exists a sequence $t_n\to \infty$ such that  $\lim_{n\to\infty}\int_0^{t_n} g(s)\,ds = 0$.   An application of Fubini's theorem permits to see that for almost every recurrent point $g$, its orbit is made of recurrent points too. Then, the set 
\begin{equation}\label{pr}
\W_1=\{g\in \W \mid g(t+\,{\cdot}\,) \;\text{is recurrent at $\infty$ for every}\;t\in\R \}\,,
\end{equation} 
which is invariant, has full measure. 
On the other hand, to the almost periodic function $f(t)$ with mean value zero and unbounded integral we can apply Theorem~3.7 in Johnson~\cite{john} which affirms that the set $\W_2\subset\W$ made up by those  $g$ for which the integral has a strong oscillatory behaviour, namely:
\begin{align*}
&\liminf_{t\to \infty} \int_0^t g(s)\,ds = -\infty\,,\quad \limsup_{t\to \infty} \int_0^t g(s)\,ds = \infty\,,\\ &\liminf_{t\to -\infty} \int_0^t g(s)\,ds = -\infty\,,\quad \limsup_{t\to -\infty} \int_0^t g(s)\,ds = \infty\,,
\end{align*}
is a residual set, that is, a topologically big set. Since clearly $\W_2\subset \W_1$, $\W_1$ is also a residual set.
\end{eje}
Connecting with this linear scalar example, it is clear that whenever the former equation is included as a decoupled $1$-dimensional subsystem of any $n$-dimensional ($n\geq 2$) linear or nonlinear system of  almost periodic ODEs or delay FDEs which is uniformly persistent, the family over the hull cannot be uniformly persistent.
\par
Having noticed that the non-robust phenomenon can well appear in higher dimensions, in the following results we   describe what is behind this situation, in the monotone and linear\//\/sublinear settings: the thing is that the uniform persistence of the individual almost periodic system is not a representative quality, as there are just a few systems in the hull with the persistence property, from both a topological and a measure theory points of view.
\par
Although the results are stated in the case of delay FDEs, they can just be rephrased for ODEs. In that case the proofs follow the same lines with some simpler arguments because of the finite-dimensional scenario. Also, the linear case is formulated in a more general context than that of the hull, because it is going to be used as a basis for the sublinear setting.
\begin{teor}\label{teor lineal n}
Let $(\W,{\cdot},\R)$ be a minimal and uniquely ergodic flow and let us consider a family of linear cooperative delay systems over $\W$,
\begin{equation}\label{linearFDE hull}
y'(t)=A(\w{\cdot}t)\,y(t)+B(\w{\cdot}t)\,y(t-1)\,,\quad \w\in\W\,,
\end{equation}
for certain continuous maps $A,\,B:\W\to M_{n}(\R)$ taking values in the set of $n\times n$-real matrices. Let us assume that for a certain $\w_0\in\W$ the corresponding system~\eqref{linearFDE hull} is uniformly persistent, that is, there exists an $m>0$ such that for any initial map $\varphi\in C([-1,0],\R^n)$,  $\varphi\gg 0$ there exists a time $t_0=t_0(\w_0,\varphi)$ such that
\begin{equation}\label{persistencia}
     y_i(t,\w_0,\varphi)\geq m \quad \text{for any }\;t\geq t_0 \;\text{ and any }\; i=1,\ldots,n \,,
\end{equation}
whereas the whole family of systems over  $\W$ is not uniformly persistent, in the sense of Definition {\rm\ref{defi persistence hull} (i)}.
Then, there exists an invariant, residual set $\W_1\subset \W$ of full measure such that for any $\w\in\W_1$, system~\eqref{linearFDE hull} is not uniformly persistent.
\end{teor}
\begin{proof}
First of all, systems~\eqref{linearFDE hull} are assumed to be cooperative, that is, all the off-diagonal entries of $A(\w)=[a_{ij}(\w)]$ and all the entries of $B(\w)=[b_{ij}(\w)]$ are nonnegative maps on $\W$. Then, the solutions of the family~\eqref{linearFDE hull} generate a linear monotone skew-product semiflow $L:\R_+\times\W\times C([-1,0],\R^n)\to \W\times C([-1,0],\R^n)$. Note that Definition~\ref{defi persistence hull} (i) of u-persistence can be naturally applied to the cooperative linear family~\eqref{linearFDE hull}.
\par
Now, once more following the procedure introduced in~\cite{noos7}, after a permutation of the variables, if necessary, we can assume that the matrix $\bar A+\bar B=[\bar a_{ij}+\bar b_{ij}]$ defined as
\begin{align*}
\bar a_{ij} &= \sup_{\w\in \W} a_{ij}(\w)\,\;\, \text{
for }\, i\not= j\,, \;\, \text{ and }\,\bar a_{ii}=0\,,
\\ \bar b_{ij}& = \sup_{\w\in \W} b_{ij}(\w)\,\;\,
\text{ for }\, i\not= j\,, \;\, \text{ and }\,\bar b_{ii}=0\,,
\end{align*}
 has the form
 \begin{equation*}
\left[\begin{array}{cccc}
\bar A_{11}+\bar B_{11}  & 0 &\ldots & 0 \\
\bar A_{21}+\bar B_{21}  & \bar A_{22} +\bar B_{22}&  \ldots& 0 \\
\vdots & \vdots &\ddots  & \vdots \\
\bar A_{k1}+\bar B_{k1} & \bar A_{k2}+\bar B_{k2} & \ldots& \bar A_{kk}+\bar B_{kk}
\end{array}\right]\,,
\end{equation*}
and the  diagonal blocks, denoted by  $\bar A_{11}+\bar B_{11},\ldots, \bar
A_{kk}+\bar B_{kk}$, of size $n_1,\ldots,n_k$ respectively
($n_1+\cdots + n_k=n$), are irreducible. For each $j=1,\ldots,k$, let $L_j$ be the linear skew-product semiflow induced on $\W\times C([-1,0],\R^{n_j})$
by the solutions of the linear systems for  $\w\in\W$ given by
the corresponding diagonal block of~\eqref{linearFDE hull},
\begin{equation}\label{subsistemasFDE}
  y'(t)=A_{jj}(\w{\cdot}t)\,y(t)+
  B_{jj}(\w{\cdot}t)\,y(t-1)\,,\quad t> 0\,.
\end{equation}
Then, $L_j$ admits a continuous separation (of type II) and its  principal spectrum reduces to the upper Lyapunov exponent of the trivial minimal set $K^j=\W\times\{0\}\subset \W\times C([-1,0],\R^{n_j})$, let us call it $\lambda_j$, because $\W$ is minimal and uniquely ergodic.
As stated in  Theorems~5.3 and~5.4 in~\cite{obsa}, which apply in this situation, the linear cooperative family~\eqref{linearFDE hull} is u-persistent if and only if $\lambda_j>0$ for any $j\in I$, for the set of indexes $I$ defined as $I=\{1\}$ if the matrix $\bar A+\bar B$ is irreducible (i.e., if $k=1$), and
\[
I=\{j\in\{1,\ldots,k\} \,\mid\, \bar A_{ji}+\bar B_{ji}=0 \text{ for any } i\not= j\}
\]
if the matrix  $\bar A+\bar B$ is reducible (i.e., if $k>1$); that is, $j\in I$ if and only if any off-diagonal block in the row of $\bar A_{jj}+\bar B_{jj}$ is null, and consequently the corresponding  system~\eqref{subsistemasFDE} is a decoupled subsystem of the total system for each $\w\in\W$. In particular this means that, as system~\eqref{linearFDE hull} for $\w_0$ is u-persistent, for any $j\in I$ system~\eqref{subsistemasFDE} for $\w_0$ is u-persistent as well.
\par
At this point, since by hypothesis the whole family is not u-persistent, at least for some $j\in I$ we have that $\lambda_j\leq 0$. It cannot be $\lambda_j<0$, as in that case all solutions of the whole family~\eqref{subsistemasFDE} would tend to $0$ as $t\to\infty$, contradicting the u-persistence for $\w_0$. Therefore it must be $\lambda_j=0$ for some $j\in I$.
\par
So, let us fix such a $j\in I$ with $\lambda_j=0$, and let $C([-1,0],\R^{n_j})=X_1(\w)\oplus X_2(\w)$ for $\w\in\W$ be the continuous splitting given by the continuous separation of the linear semiflow  $L_j(t,\w,\varphi)=(\w{\cdot}t,\Phi_j(t,\w)\,\varphi)$, $(t,\w,\varphi)\in \R_+\times \W\times C([-1,0],\R^{n_j})$. Recall that $X_1(\w)=\spa\{ v(\w)\}$  determines a one-dimensional invariant subbundle, with $v:\W\to C([-1,0],\R^{n_j})$ continuous and such that  $v(\w)\gg 0$ and $\|v(\w)\|=1$ for any $\w\in \W$. In particular, $0<v_i(\w)(s)\leq 1$ for any $\w\in\W$, any component $i=1,\ldots,n_j$ and any $s\in [-1,0]$.
\par
Now, we follow the arguments used in Proposition~5.1 (iii) in Calzada et al.~\cite{calz} in a quasi-periodic setting, which remain valid here. All the details are explained in that paper. For the norm $\|\varphi\|_2=\,\left(\|\varphi(0)\|^2+\int_{-1}^{0}\|\varphi(s)\|^2\,ds\right)^{1/2}$ in the space $Y=L^2([-1,0],\R^{n_j},\mu_0)$ for the measure $\mu_0=\delta_0+l$, where $\delta_0$ is the Dirac measure concentrated at $0$ and $l$ is the Lebesgue measure on $[-1,0]$, we consider the normalized functions $\wit v(\w)=v(\w)/\|v(\w)\|_2$ and we recall that  there is a $\delta>0$ such that $\delta\leq \|v(\w)\|_2$ for any $\w\in \W$.  Then, we consider the map $\wit c(t,\w)$  satisfying $\Phi_j(t,\w)\,\wit v(\w)=\wit c(t,\w)\,\wit v(\w{\cdot}t)$ for any $t\geq 0$ and $\w\in\W$, which can be extended to the whole line  fulfilling  the linear cocycle identity $\wit c(t+s,\w)=\wit c(t,\w{\cdot}s)\,\wit c(s,\w)$
for any $t,s\in \R$ and any $\w\in \W$. Besides, the expression $\wit a(\w)=\left.\frac{d}{dt} \log \wit c(t,\w)\right|_{t=0}$ defines a continuous map on $\W$. We remark that the $L^2$-norm has been taken in order to have nice differentiability properties on the scalar map $\log \wit c(t,\w)$ associated with the continuous separation. Moreover, as shown in~\cite{calz}, the Lyapunov exponent of each $\w\in\W$ can be calculated as
\[
\lambda_j(\w)= \limsup_{t\to\infty} \frac{\log \wit c(t,\w)}{t}\,,
\]
and arguing as in the proof of Theorem~\ref{teor Nicholson system}, in what refers to the application of the theory by Sacker and Sell~\cite{sase}, we can conclude that for any $\w\in\W$, the upper Lyapunov exponent $\lambda_j=\lambda_j(\w)$, and so
\[
\lambda_j=\lim_{t\to \infty} \frac{\log \wit c(t,\w)}{t}=\lim_{t\to \infty}
\frac{1}{t} \int_0^t  (\log \wit c(s,\w))'ds=\lim_{t\to \infty}
\frac{1}{t} \int_0^t \wit a(\w{\cdot}s)\,ds=\int_{\W} \wit a\,d\mu\,,
\]
where $\mu$ is the unique ergodic measure on $\W$ and  Birkhoff ergodic theorem has been applied to the map $\wit a\in C(\W)$ in the last equality.
\par
As a consequence, since $\lambda_j=0$, then $\wit a\in C(\W)$ has zero mean value and $\wit c(t,\w)$ is the scalar linear cocycle giving the solutions $h(t,\w,y_0)=y_0\,\wit c(t,\w)$ $(y_0\in \R)$ of the family of recurrent scalar linear equations for $\w\in\W$,
\begin{equation}\label{ec h}
h'(t)=\wit a(\w{\cdot}t)\,h(t)\,,\quad   t\in \R\,.
\end{equation}
\par
Then, arguing as in Example~\ref{ejemplo}, after the zero mean value of  $\wit a$ we can deduce that there is an invariant set $\W_1\subset \W$ of full measure formed by recurrent points at $\infty$. In particular,  for any $\w\in\W_1$,   $\lim_{n\to\infty}\int_0^{t_n} \wit a(\w{\cdot}s)\,ds = 0$ for a sequence $t_n\to \infty$. Since $\Phi_j(t,\w)\,\wit v(\w)=\wit c(t,\w)\,\wit v(\w{\cdot}t)$ for any $t\geq 0$ and $\w\in\W$,
we get that for each $\w\in\W_1$, $\Phi_j(t_n,\w)\,\alpha\,\wit v(\w)=\wit c(t_n,\w)\,\alpha\,\wit v(\w{\cdot}t_n)\to\alpha\,\wit v(\w{\cdot}t_n)$, as $n\to\infty$ for a sequence $t_n=t_n(\w)\to \infty$ and for any $\alpha>0$, which precludes the property of u-persistence for system~\eqref{subsistemasFDE}, and consequently also for system~\eqref{linearFDE hull} whenever $\w\in\W_1$.
\par
Besides, the set $\W_1$ is also residual. To see it, note that the u-persistence of system~\eqref{linearFDE hull} for $\w_0$ implies for the decoupled subsystem~\eqref{subsistemasFDE} that, for each $y_0>0$, given the initial map $y_0\,\wit v(\w_0)\gg 0$, there is a $t_0=t_0(y_0\,\wit v(\w_0))$ such that $\Phi_j(t,\w_0)\,y_0\,\wit v(\w_0)=\wit c(t,\w_0)\,y_0\,\wit v(\w_0{\cdot}t)\geq \bar m$ for any $t\geq t_0$, for the map $\bar m\in C([-1,0],\R^{n_j})$ with all components identically equal to $m$. Therefore, for  $\delta>0$ such that $\delta\leq \|v(\w)\|_2$ for any $\w\in \W$, it holds that for any $t\geq t_0$,
\[
\frac{1}{\delta}\,y_0\,\wit c(t,\w_0)\geq y_0\,\wit c(t,\w_0)\,\frac{v_i(\w_0{\cdot}t)}{\|v(\w_0{\cdot}t)\|_2}\geq m\,,
\]
where any component of $v(\w_0{\cdot}t)$ can been chosen. That is, $y_0\,\wit c(t,\w_0)\geq m\,\delta$ for any $t\geq t_0$, and this is exactly u-persistence for the scalar linear equation~\eqref{ec h} for $\w_0$, with eventual positive lower bound $m\,\delta$. Proposition~\ref{prop-linear ap} then asserts that $\lim_{t\to \infty}\int_0^t \wit a(\w_0{\cdot}s)\,ds=\infty$, and in particular this means unbounded integral of $\wit a$ along the orbit of $\w_0$. This, together with the zero mean value of $\wit a$, permits to apply once more Theorem~3.7 in~\cite{john} to conclude that the set $\W_2\subset\W$ composed by those  $\w$ for which the integral $\int_0^t \wit a(\w{\cdot}s)\,ds$ has a strong oscillatory behaviour, in the precise terms explained in  Example~\ref{ejemplo}, is a residual set. To finish, since clearly $\W_2\subset \W_1$, we conclude that $\W_1$ is also a residual set.
\end{proof}
Finally, we consider cooperative and sublinear systems.
Since we are especially interested in systems which are modelling real world biological processes, it is quite natural to assume that $0$ is a solution and that solutions starting with nonnegative initial data keep being nonnegative while defined. This is so if $f(t,0)=0$ for any $t\in\R$ in the ODEs case and $f(t,0,0)=0$ for any $t\in\R$ in the delay case, together with the cooperative condition. In this situation, uniform persistence is considered in the sense of~\eqref{persistencia} for an individual delay system, and in the sense of Definition~\ref{defi persistence hull} (i) for the induced families over the hull.
\par
More precisely, we are assuming the following hypotheses either for system~\eqref{ode} $y'(t)=f(t,y(t))$ or for the delay system~\eqref{delay} $y'(t)=f(t,y(t),y(t-1))$.
\begin{itemize}
\item[(H1)] The function $f$ defining the system is uniformly almost periodic, it satisfies
the regularity and admissibility conditions stated in Section~\ref{sec-preli}, and the identically null map is a solution.
\item[(H2)] The system is cooperative and sublinear.
\end{itemize}
Recall that, under the assumption that $0$ is a solution, if the system is cooperative and concave, it is also sublinear, so that this case is also included. Once more, we only write the result for delay equations.
\begin{teor}\label{teor sublineal}
Let us consider a finite-delay FDEs nonlinear system~\eqref{delay} under assumptions {\rm (H1)} and {\rm (H2)} and let $\tau$ be the monotone and sublinear skew-product semiflow of class $\mathcal C^1$ defined on $\W\times C_+([-1,0],\R^{n})$ by the solutions $y(t,\w,\varphi)$ of the family of systems over the hull~\eqref{delayfamily} $y'(t)=F(\w{\cdot}t,y(t),y(t-1))$, $\w\in\W$.
Then:
\begin{itemize}
\item[(i)] The family of systems~\eqref{delayfamily}  is uniformly persistent  if and only if the family of linearized  systems along the null solution is   uniformly persistent.
\item[(ii)] If  system~\eqref{delay} is uniformly persistent, whereas the family of systems over the hull~\eqref{delayfamily} is not uniformly persistent, then there exists an invariant, residual set $\W_1\subset \W$ of full measure such that for any $\w\in\W_1$, system~\eqref{delayfamily} is not uniformly persistent.
\end{itemize}
If we assume further that there exists a map $\varphi_0\gg 0$ such that the solution $y(t,\varphi_0)$ of system~\eqref{delay} with initial value $\varphi_0$ is bounded, then:
\begin{itemize}
\item[(iii)] System~\eqref{delay} is uniformly persistent  if and only if the family of systems~\eqref{delayfamily} is uniformly persistent.
\item[(iv)] The uniform persistence of system~\eqref{delay} can be characterized by a set of computable Lyapunov exponents determined by the structure of its linearized system along $0$.
\end{itemize}
\end{teor}
\begin{proof}
First of all, the same proof as that of Proposition~2.3 in~\cite{nuos3} for 2-dimensional systems permits to conclude that the induced semiflow is globally defined.
\par
(i) The transfer of the u-persistence from the linearized family to the nonlinear family is a direct consequence of Theorem~5.4 in~\cite{obsa} for general recurrent and  cooperative systems. Conversely, one just applies a comparison of solutions argument having in mind the inequality~\eqref{mayorante lineal} which also holds in the sublinear setting. More precisely, let us write down the family of linearized systems along the null solution, which is of the form~\eqref{linearFDE hull} for the matrix-valued continuous maps on $\W$ defined by $A(\w)=D_y F(\w,0,0)$ and $B(\w)=D_w F(\w,0,0)$, where we have written $F=F(\w,y,w)$. Then, $y_t(\w,\varphi) \leq D_\varphi y_t(\w,0)\,\varphi$, for any $\w\in\W$, $\varphi\geq 0$ and $t\geq 0$,
and the functions $z(t,\w,\varphi)=(D_\varphi y_t(\w,0)\,\varphi)(0)$ $(t\geq 0)$ are precisely the solutions of the linearized family. Therefore, the u-persistence of the sublinear family below forces the u-persistence of the linearized family above.
\par
(ii) As just mentioned, since the family over the hull~\eqref{delayfamily} is not u-persistent, neither is u-persistent the associated family~\eqref{linearFDE hull} of linearized systems along $0$, described in (i). On the other hand, calling $\w_0=f\in \W$ the element providing the initial system~\eqref{delay}, the u-persistence of system~\eqref{delay} implies that of the linearized system~\eqref{linearFDE hull} for $\w_0$.
Then, we can apply Theorem~\ref{teor lineal n} to assert that there exists an invariant, residual set $\W_1\subset \W$ of full measure such that for any $\w\in\W_1$, system~\eqref{linearFDE hull} is not u-persistent. Once more by the inequality in (i), this implies that neither is system~\eqref{delayfamily}  u-persistent for $\w\in\W_1$, and we are done.
\par
(iii) First of all, note that the existence of a bounded solution under the assumption of u-persistence of system~\eqref{delay} completely precludes the linear case.
Now, as it could not be otherwise, it is immediate that the u-persistence goes nicely from the family to a particular system.
\par
Conversely, if system~\eqref{delay} is u-persistent we apply to $\tau$ the dynamical description developed in  N\'{u}\~{n}ez  et al.~\cite{nuos2} for general monotone and sublinear skew-product semiflows.  Arguing as in the proof of Theorem~\ref{teor Nicholson pers} for Nicholson systems, for $\w_0=f$ we consider  the  orbit of $(\w_0,\varphi_0)$ which is bounded. Then, one can consider  its omega-limit set, which contains a minimal set $K$ necessarily lying on the zone $\W\times\{\varphi\in C_+([-1,0],\R^n)\mid \varphi\geq \bar m\}$, for the map $\bar m$ whose components are identically equal to $m$, the constant involved with the property of u-persistence of system~\eqref{delay}. In other words, there is a strongly positive minimal set for $\tau$. Thus, Theorem~3.8 in~\cite{nuos2} asserts that the dynamics suits one the following three cases: the so-called case A1 when $K$ is the unique minimal set strongly above $0$; case A2 when there are infinitely many minimal sets strongly above $0$ and, among them, there exists one $K^-$ which is the lowest one; or case A3 when there are infinitely many minimal sets strongly above $0$ but there is not a lowest one.
\par
Exactly as in the proof of Theorem~\ref{teor Nicholson pers}, case A3 is discarded thanks to the u-persistence of system~\eqref{delay}, and in both cases A1 and A2 the family~\eqref{delayfamily} turns out to be u-persistent, due to the attracting properties enjoyed by the minimal sets $K$ and $K^-$ respectively (see~\cite{nuos2} for more details).
\par
(iv) The statement of this item has not been written more precisely in order not to make the paper too long, but the reader is referred to Theorem~\ref{teor Nicholson system} for a very detailed statement in the same line in the case of almost periodic Nicholson systems.
\par
The key is in (i) and (iii) together, saying that the property of u-persistence for system~\eqref{delay} is equivalent to that of the family of systems~\eqref{delayfamily} and also to that of the family of linearized systems along $0$. Besides, for the associated linear family~\eqref{linearFDE hull} described in (i), which is cooperative, the property of u-persistence has been characterized in terms of a precise set of upper Lyapunov exponents $\{\lambda_j\mid j\in I\}$, as it has been explained in detail in the proof of Theorem~\ref{teor lineal n}.
\par
Once more, the theory by Sacker and Sell~\cite{sase} applies to the $1$-dimensional linear skew-product semiflow associated with the dynamics in the $1$-dimensional invariant subbundle given by the continuous separation of $L_j$ for each $j\in I$ (for $L_j$ defined in the proof of Theorem~\ref{teor lineal n}). As a consequence, for each $j\in I$ the exponent $\lambda_j=\lambda_j(\w)$ for any $\w\in\W$, and in particular $\lambda_j=\lambda_j(\w_0)$ for $\w_0=f$, the element providing the initial system~\eqref{delay}. So that, in the end, the property of u-persistence is characterized in terms of a set of Lyapunov exponents of some lower-dimensional linear systems chosen from the structure of the linearized system along $0$.
\end{proof}
To end the paper, we make a couple of remarks. First, a more general recurrent time variation rather than almost periodicity may be admitted  in the statement of Theorem~\ref{teor sublineal}~(i) and~(iii), as we just need $\W$ to be minimal (for instance, see~\cite{noos7}). However, the unique ergodicity of $\W$ is also needed in both (ii) and (iv).  Second and last, the fact that there exists a bounded solution is sometimes implicitly required in the literature by assuming the existence of an upper-solution (for instance, see Zhao~\cite{zhaox} and Mierczy{\'n}ski and Shen~\cite{mish}) or by asking the system to be dissipative.


\begin{thebibliography}{99}
\bibitem{aman} {\sc H. Amann}, Fixed point equations and
        nonlinear eigenvalue problems in ordered Banach spaces,
        {\em SIAM Rev.\/} {\bf 18} (1976), 620--709.
\bibitem{bebrid} {\sc L. Berezansky, E. Braverman, L. Idels}, Nicholson's blowflies differential equations revisited: main results and open problems, {\it  Appl. Math. Model.} {\bf 34} (2010), 1405--1417.
\bibitem{brva} {\sc D. Breda, E. Van Vleck}, Approximating Lyapunov exponents and Sacker-Sell spectrum for retarded functional differential equations, {\em Numer. Math} \textbf{126} (2014), 225-–257.
\bibitem{calz} {\sc J.A. Calzada, R. Obaya, A.M. Sanz}, Continuous separation for monotone skew-product semiflows: From theoretical to numerical results, {\em Discrete Contin. Dyn. Syst. Series B}
        \textbf{20} (3) (2015), 915--944.
\bibitem{capa}{\sc V.~Capasso},
        {\em Mathematical Structures of Epidemic Systems},
        Lecture Notes in Biomathematics {\bf 97}, Springer-Verlag, Berlin, 1993.
\bibitem{comi}  {\sc C.C. Conley, R.K. Miller}, Asymptotic stability without uniform stability: almost periodic coefficients, {\em J. Differential Equations} {\bf 1} (1965), 333--336.
\bibitem{elli} {\sc R. Ellis},         \textit{Lectures on Topological Dynamics\/},
        Benjamin, New York, 1969.
\bibitem{faria}  {\sc T. Faria}, Asymptotic behaviour for a class of delayed cooperative models with patch structure,  {\em Discrete Contin. Dyn. Syst. Series B}         \textbf{18} (6) (2013), 1567--1579.
\bibitem{faos}  {\sc T. Faria, R. Obaya, A.M. Sanz},   Asymptotic behaviour for a class of non-monotone delay differential systems with applications. {\em J. Dynamics Differential Equations} (2017). https://doi.org/10.1007/s10884-017-9572-8.
\bibitem{faro}  {\sc T. Faria, G. R\"{o}st}, Persistence, permanence and global stability for an
n-dimensional Nicholson system, {\em J. Dynamics Differential Equations} \textbf{26}
        (2014), 723--744.
\bibitem{gubl} {\sc W.S.C. Gurney, S.P. Blythe, R.M. Nisbet}, Nicholson's blowflies revisited, {\em Nature} \textbf{287} (1980), 17--21.
\bibitem{have} {\sc J.K. Hale, S.M. Verduyn Lunel},
        \textit{Introduction to Functional Differential
        Equations}, Applied Mathematical Sciences {\bf 99},
        Springer-Verlag, Berlin, Heidelberg, New York, 1993.
\bibitem{hosc} {\sc J. Hofbauer, S.J. Schreiber},
        To persist or not to persist?,
        {\em Nonlinearity}  \textbf{17} (2004), 1393--1406.
\bibitem{john} {\sc R. Johnson}, Minimal functions with unbounded integral,
        {\em Israel J. Math.} {\bf 31} (1978), 133--141.
\bibitem{lime}   {\sc X. Liu, J. Meng}, The positive almost periodic solution for Nicholson-type delay systems with linear harvesting term, {\em Appl. Math. Model.} \textbf{36} (2012),     3289--3298.
\bibitem{magl} {\sc M.C. Mackey, L. Glass}, Oscillation and chaos in physiological control systems, {\em Science} \textbf{197} (4300) (1977), 287--289.
\bibitem{mish} {\sc J. Mierczy{\'n}ski, W. Shen}, Lyapunov
        exponents and asymptotic dynamics in random Kolmogorov models,
        {\em J. Evol. Equ.} {\bf 4} (2004), 371--390.
\bibitem{nich} {\sc A.J. Nicholson}, An outline of the dynamics of animal populations, {\em Austral. J. Zool.} \textbf{2} (1954), 9--65.
\bibitem{noob1}  {\sc S. Novo, R. Obaya}, Non-autonomous functional differential equations and applications. {\em Stability and Bifurcation for non-autonomous
        differential equations}, 185--264, Lecture Notes in Math.
            \textbf{2065}, Springer-Verlag, Berlin, Heidelberg, 2013.
\bibitem{nooslondon} {\sc S. Novo, R. Obaya, A.M. Sanz},  Attractor minimal sets for non-autonomous delay functional differential equations with applications for neural networks, {\em Proc. R. Soc. Lond. Ser.
A Math. Phys. Eng. Sci.} {\bf 461} (2005),  2767-–2783.
\bibitem{noos6} {\sc S. Novo, R. Obaya, A.M. Sanz}, Topological dynamics for monotone skew-product semiflows with applications,  {\em J. Dynamics Differential Equations} {\bf 25} (4) (2013), 1201--1231.
\bibitem{noos7} {\sc S. Novo, R. Obaya, A.M. Sanz}, Uniform persistence and upper Lyapunov exponents for monotone skew-product semiflows, {\em Nonlinearity} \textbf{26} (9) (2013), 2409--2440.
\bibitem{nuos2} {\sc C. N\'{u}\~{n}ez, R. Obaya, A.M. Sanz},
       Minimal sets in monotone and sublinear skew-product semiflows I:
       The general case, {\it J. Differential Equations\/} \textbf{248} (2010),
       1879--1897.
\bibitem{nuos3} {\sc C. N\'{u}\~{n}ez, R. Obaya, A.M. Sanz}, Minimal sets in monotone and
    sublinear skew-product semiflows II: Two-dimensional systems of differential
    equations,  {\it J. Differential Equations\/} \textbf{248} (2010), 1899--1925.
\bibitem{nuos4} {\sc C. N\'{u}\~{n}ez, R. Obaya, A.M. Sanz},
        Minimal sets in monotone and concave skew-product semiflows I:
        A general theory, {\it J. Differential Equations\/} \textbf{252} (10) (2012),
       5492--5517.
\bibitem{obsa} {\sc R. Obaya, A.M. Sanz}, Uniform and strict persistence in monotone skew-product semiflows with applications to non-autonomous Nicholson systems, {\it J. Differential Equations\/} \textbf{261} (2016),
       4135--4163.
\bibitem{poin} {\sc H. Poincar\'{e}},  Sur les courbes  d\'{e}finies par les \'{e}quations diff\'{e}rentielles IV, {\it Journal de Math\'{e}matiques Pures et Appliqu\'{e}es} {\bf 2} (1886), 151--217.
\bibitem{pote} {\sc P. Pol\'{a}\v{c}ik, I. Tere\v{s}\v{c}\'{a}k}, Exponential
        separation and invariant bundles for maps in ordered Banach spaces
         with applications to parabolic equations, {\em J. Dynamics
        Differential Equations} {\bf 5} No. 2 (1993), 279--303.
\bibitem{sase} {\sc R.J. Sacker, G.R. Sell}, A spectral theory for
        linear differential systems, {\em J. Differential Equations}
               {\bf 27} (1978), 320--358.
\bibitem{shyi5} {\sc W. Shen, Y. Yi},
        Convergence in almost periodic Fisher and Kolmogorov models,
        {\em J. Math. Biology} {\bf 37} (1998), 84--102.
\bibitem{shyi} {\sc W. Shen, Y. Yi}, Almost Automorphic and Almost
        Periodic Dynamics in Skew-Product Semiflows, {\em Mem. Amer. Math. Soc.}
        {\bf 647}, Amer. Math. Soc., Providence  1998.
\bibitem{shne} {\sc Ya. Shneiberg}, Zeros of integrals along trajectories of
        ergodic systems, {\em Funktsional. Anal. i Prilozhen.} {\bf 19}
        (2) (1985), 92--93.
\bibitem{smit} {\sc H.L. Smith},
        {\em Monotone Dynamical Systems. An Introduction to the Theory
        of Competitive and Cooperative Systems},
        Amer. Math. Soc., Providence, 1995.
\bibitem{smth} {\sc H.L. Smith, H.R. Thieme}, {\em Dynamical Systems and Population Persistence}, Graduate Studies in Mathematics, {\bf 118}. Amer. Math. Soc., Providence, 2011.
\bibitem{wang}   {\sc L. Wang}, Almost periodic solution for Nicholson's blowflies model with patch structure and linear harvesting terms, {\em Appl. Math. Model.} \textbf{37} (2013),     2153--2165.
\bibitem{zhaox} {\sc X.-Q. Zhao},
        Global attractivity in monotone and subhomogeneous almost periodic systems,
        {\em J. Differential Equations\/} {\bf 187} (2003), 494--509.
\bibitem{zhaox2} {\sc X.-Q. Zhao},
        {\em Dynamical Systems in Population Biology},
        CMS Books in Mathematics, Springer-Verlag, New York, 2003.
\end{thebibliography}
\end{document}